\documentclass[11pt,a4paper]{article}
\usepackage{graphics}
\usepackage{indentfirst}
\usepackage{fancyhdr}
\usepackage{mathrsfs}
\usepackage{abstract}
\usepackage{amsfonts,amssymb,amsmath}
\usepackage{amsthm,thmtools}                

\newtheorem{thm}{Theorem}[section]
\newtheorem{lem}[thm]{Lemma}

\newtheorem{pro}[thm]{Proposition}
\newtheorem{defi}[thm]{Definition}
\newtheorem{proA}{Proposition}
\newtheoremstyle{rem}{10pt}{10pt}{\rmfamily}{}{\bfseries}{.}{.5em}{} 
\theoremstyle{rem}
\newtheorem{rem}[thm]{Remark}     

\begin{document}
	\title{Almost sure scattering for the  nonlinear Klein-Gordon equations with  Sobolev critical power}
	\author{Jie Chen,\qquad Baoxiang Wang}
	\date{}
	\maketitle
	
	\begin{abstract}
		In this paper, we study the almost sure scattering for the Klein-Gordon equations with Sobolev critical power. We obtain the almost sure scattering  with random initial data in $H^s\times H^{s-1}, \frac{11}{12}<s<1$ for $d=4$, $\frac{15}{16}<s<1$ for $d=5$. We use the	induction on scales and bushes argument in \cite{Bringmann1} where the model equation is wave equation. For $d=5$, we use the mass term of the Klein-Gordon equation to obtain the control of the increment of energy in the process of induction on scales. 
	\end{abstract}
	\section{Introduction}
	We consider the Cauchy problem for the defocusing nonlinear Klein-Gordon equations with Sobolev critical power in space dimensions $d = 4,5$ with real intial data. $\mathcal{H}^s:= H^{s}\times H^{s-1}$.
	\begin{equation}\label{aim0}
		\left\{
		\begin{array}{l}
			u_{tt} - \Delta u + u + u^\frac{d+2}{d-2} = 0, \quad (t,x)\in \mathbb{R}\times \mathbb{R}^d;\\
			(u,u_t)\big|_{t=0} = (u_0,u_1)\in \mathcal{H}^s(\mathbb{R}^d).
		\end{array}
		\right.
	\end{equation}

	Sufficiently regular solutions of (\ref{aim0}) conserve the energy
	\begin{equation*}
		E(u(t), u_t(t)) := \int_{\mathbb{R}^d} \frac{1}{2}u_t(t)^2+\frac{1}{2} |\nabla u(t)|^2+\frac{1}{2}u(t)^2+\frac{d-2}{2d}u(t)^\frac{2d}{d-2}dx.
	\end{equation*}
	We call equation (\ref{aim0}) with Sobolev critical power since the nonlinear term $\int_{\mathbb{R}^d} u^{\frac{2d}{d-2}}dx$ can be controlled by  $\int_{\mathbb{R}^d}|\nabla u|^2 dx$ due to Sobolev embedding inequality.
	
	The Cauchy problem for equation (\ref{aim0}) with initial data in the energy space $\mathcal{H}^1$ is well-understood. The properties of equation (\ref{aim0}) resemble the energy critical wave equation. There exists a vast body of literatrue related to the energy critical wave equation. For the global existence of (\ref{aim0}), see \cite{Grillakis1992Regularity}, \cite{Kapitanski}, \cite{1994Well}.  By the Bourgain's energy induction argument \cite{Bourgaine}, in \cite{NakanishiKG}, Nakanishi established the relative complete results in energy space 
	considering the global time-space bounds and scattering.
	
	\begin{thm}[Nakanishi \cite{NakanishiKG}]\label{Nakanishi}
		For $d = 4, 5$, given $(u_0,u_1)\in \mathcal{H}^1$, 
		There exists a unique global solution $(u,u_t)\in C(\mathbb{R}, \mathcal{H}^1)$, $u\in L^\frac{d+2}{d-2}(\mathbb{R}, L^\frac{2(d+2)}{d-2})$ of equation (\ref{aim0}), and
		\begin{equation*}
			\|u\|_{L^\frac{d+2}{d-2}(\mathbb{R}, L^\frac{2(d+2)}{d-2})}\leq C(E(u_0,u_1)).
		\end{equation*}
		$u$ scatters to a solution of the linear Klein-Gordon equation. It means that there exists $(u_0^{\pm}, u_1^{\pm})\in \mathcal{H}^1$ such that
		\begin{equation*}
			\lim_{t\rightarrow \pm \infty}\|(u(t),u_t(t))-K(t)(u_0^\pm,u_1^\pm)\|_{\mathcal{H}^1}  = 0.
		\end{equation*}
		Here, $K(t)(u_0^\pm,u_1^\pm) = (\pi_1 K(t)(u_0^\pm,u_1^\pm), \partial_t(\pi_1 K(t)(u_0^\pm,u_1^\pm)))$, and
		\begin{align*}
			\pi_1 K(t)(u_0^\pm,u_1^\pm) = \cos(t\langle \nabla \rangle)u_0^\pm + \frac{\sin (t\langle \nabla \rangle)}{\langle \nabla \rangle}u_1^\pm.
		\end{align*}
	\end{thm}

	\begin{rem}
		In \cite{NakanishiKG}, the space dimensions are $d\geq 3$. The global time-space bounds are related to some Besov spaces in Proposition 5.1 \cite{NakanishiKG}. For $d = 3, 4, 5$, as claimed in \cite{NakanishiKG}, we can obtain the global $L^\frac{d+2}{d-2}_t L^\frac{2(d+2)}{d-2}_x$ bounds by Strichartz estimates.
	\end{rem}

	\begin{rem}
		In \cite{Christ2003Ill}, Christ-Colliander-Tao showed that the energy critical wave equations exhibit norm inflation. As claimed in \cite{Christ2003Ill}, for the examples constructed for wave equation, the mass term plays no significant role after rescaling since the examples are ``high-frequency". Suppose $s<1$. 
		By Theorem 6, 8 in \cite{Christ2003Ill}, given $\varepsilon>0$, there exists Schwartz functions $u_0$, $\|u_0\|_{{H}^s}<\varepsilon$, for some $0<t<\varepsilon$, such that the solution of (\ref{aim0}) with intial data $(u_0,0)$ satisfies $\|(u(t),u_t(t))\|_{\mathcal{H}^s}>\frac{1}{\varepsilon}$ . By finite speed of propagation, there exists $u_0\in C^\infty\cap H^s$, such that the corresponding solution $u(t,x)\in C^\infty([0,1]\times \mathbb{R}^d)$ with initial data $(u_0,0)$ satisfies $$\|u(t)\|_{H^s} = \infty, ~\forall ~t\in(0,1].$$
	\end{rem}

	Although the nonlinear Klein-Gordon equation (\ref{aim0}) is ill-posed in $\mathcal{H}^s$ for $s<1$, it is sometimes possible to construct ``unique" solutions by randomizing the initial data. The study of dispersive partial differential equations via a probabilistic approach was initiated by Bourgain
	\cite{bourgain1994,bourgain1996invariant} for the nonlinear Schr\"{o}dinger equation on $\mathbb{T}$ in dimensions $1$ and $2$. Then, Burq-Tzvetkov \cite{burq2008random1,burq2008random2} explored such problems in the context of the cubic nonlinear wave equation on a $3D$ compact Riemannian manifold. There exists a  vast body of literature where probabilistic tools are used to study nonlinear dispersive equations in scaling super-critical regimes. See \cite{BOP1,BOP4,BOP2,BOP3,OP1,pocovnicu2017probabilistic,B_nyi_2019,chen2020random} and references therein. For the scattering results, see \cite{DLM,killip2019almost,dodson2019almost,Bringmann_2020}. For the long time behavior in the context of fucusing wave equation, see \cite{kenig2019focusing}. See also \cite{ZhangFang} for the similar problem in the context of the dissipative equation. 
	
	\subsection{Randomization procedure}  

	
	We use the randomization based on the the uniform decomposition of the phase space. The randomization is similar to the one used in Bringmann \cite{Bringmann1} which was referred to as microlocal randomization.
	
	Let $\varphi\in C^\infty_0(\mathbb{R}^d), ~0\leq \varphi \leq 1$ satisfying
	\begin{equation*}
		\varphi|_{[-\frac{1}{4},\frac{1}{4}]^d} \equiv 1 ,\quad \varphi|_{([-1,1]^d)^c} \equiv 0, \quad \varphi_k(\xi) := \varphi(\xi-k),\quad\sum\limits_{k\in \mathbb{Z}^d}\varphi_k = 1.
	\end{equation*}
	Then, define the Fourier multiplier operators $P_k = \mathscr{F}^{-1}\varphi_k \mathscr{F}, k \in \mathbb{Z}^d$, where we use
	$\mathscr{F} f(\xi):=\hat{f}(\xi) = \int_{\mathbb{R}^d} f(x) e^{-ix\cdot \xi} dx$.
	
	We use two sequences of independent, real valued, mean zero, uniformly sub-Gaussian random variables $\{X_k\}, \{Y_l\}$ to randomize the space and frequence separately. See subsection \ref{randommmmm} for more descriptions. Suppose $\{X_k\}, \{Y_l\}$ random variables on probability spaces $\Omega_1, \Omega_2$ satisfy the above assumptions.
	For $(f,g)\in H^s\times H^{s-1}$, 	$\omega = (\omega_1,\omega_2)\in \Omega_1\times \Omega_2:=\Omega$. Define
	\begin{equation}\label{initialdatacondition}
		f^\omega = \sum_{k,l\in \mathbb{Z}^d} X_{k}(\omega_1) P_k(Y_l(\omega_2)\varphi_l f),~g^\omega = \sum_{k,l\in \mathbb{Z}^d} X_{k}(\omega_1) P_k({Y}_l(\omega_2)\varphi_l g),
	\end{equation}


	\begin{rem}
		In \cite{BOP1,BOP4,BOP2,BOP3,OP1,pocovnicu2017probabilistic,DLM,killip2019almost,dodson2019almost}, the randomization considered there is based on the uniform decomposition of 
		the frequency space, 
		$$f^\omega = \sum_{k\in \mathbb{Z}^d}X_k(\omega)P_kf.$$
		This randomization was referred to as ``Wiener randomization" in \cite{BOP1, BOP4,OP1}. This terminology was 
		closely related to
		the modulation spaces introduced by H. Feichtinger \cite{Feitingermodulation}. 
		In \cite{WangZG}, Wang-Zhao-Guo first applied the frequency uniform decomposition
		operators $P_k$ to study nonlinear evolution equations.
		See \cite{WangHHG} for more explanations of the frequency uniform decomposition techniques. 
		Recently, in \cite{Chen_2020},
		Chen-Wang-Wang-Wong first applied the uniform decomposition of the phase space to study dissipative nonlinear evolution pseudo-differential equations. The decomposition for a function $f$ in \cite{Chen_2020} is
		$$f = \sum_{k,l\in\mathbb{Z}^d}\varphi_l P_k f.$$
		It is different from the decomposition in this paper where we decompose the physical  space $\mathbb{R}^d_x$ first. However, both decompositions are based on the uniform decomposition of the phase space.
	\end{rem}
	
	\begin{rem}
		The more general assumptions on $\{X_k\}, \{Y_l\}$ are mean zero, independent, symmetric, uniformly sub-Gaussian. In \cite{Bringmann1}, Bringmann randomizes phase-space blocks $\{P_k(\varphi_l f)\}$ by $X_{k,l}(\omega)$. The essential use of the randomization there is to consider the randomization of space and frequency separately. Thus, we can extend the results in this paper to the random variables considered in \cite{Bringmann1}  by the similar argument.
	\end{rem}
	
	\subsection{Main result}
	We use the ideas of Bringmann in \cite{Bringmann1} to obtain the corresponding results for the nonlinear Klein-Gordon equations with Sobolev critical power. For $(f^\omega,g^\omega)$ defined by (\ref{initialdatacondition}), consider the Cauchy problem
	\begin{equation}\label{aim}
		\left\{
		\begin{array}{l}
			u_{tt} - \Delta u + u + u^\frac{d+2}{d-2} = 0, \quad (t,x)\in \mathbb{R}\times \mathbb{R}^d;\\
			u(0) = f^\omega, \quad u_t(0) = g^\omega.
		\end{array}
		\right.
	\end{equation}

	Recall the definition of $K(t)$ in Theorem \ref{Nakanishi}. The main result in this paper is :	
	\begin{thm}\label{finallll}
	For $d=4,\frac{11}{12}<s<1$ or $d = 5, \frac{15}{16}<s<1$,
	$(f,g)\in \mathcal{H}^s$  real.  Then, for $\omega \in \Omega, ~ a.e$, there exists a unique global solution $u$ of $(\ref{aim})$ $s.t$. 
	\begin{equation*}
		(u(t),u_t(t))- K(t)(f^\omega,g^\omega)\in C(\mathbb{R},\mathcal{H}^1), \quad u\in L^\frac{d+2}{d-2}(\mathbb{R}, L^\frac{2(d+2)}{d-2}).
	\end{equation*}
	\begin{equation*}
		u(t) = \cos(t\langle \nabla \rangle)f^\omega+\frac{\sin t\langle \nabla \rangle}{\langle \nabla \rangle}g^\omega - \int_{0}^{t}\frac{\sin(t-s)\langle \nabla \rangle}{\langle \nabla \rangle}|u|^\frac{4}{d-2}u(s)ds
	\end{equation*}
	There exists $(u^\infty_0,u^\infty_1)\in \mathcal{H}^1$ such that
	\begin{equation*}
		\lim_{t\rightarrow \infty}\|(u(t),u_t(t))-K(t)(f^\omega,g^\omega)-K(t)(u_0^\infty,u_1^\infty)\|_{\mathcal{H}^1} = 0.
	\end{equation*}
	The similar statement also holds as $t\rightarrow -\infty$.
    \end{thm}

	By the norm inflation phenomenon for the initial data in $\mathcal{H}^s, ~s<1$, we can not solve the equation (\ref{aim0}) by iteration directly. The strategy to overcome this difficulty  is to use the  Da Prato-Debussche trick \cite{Dapratotrick}, which was used generally for this kind of problem. Let $F_{\geq N}^\omega(t) := \pi_1 K(t)(f_{\geq N}^\omega, g_{\geq N}^\omega)$. Then, for $v(t) = u(t) - F_{\geq N}^\omega(t)$, $v$ satisfies the equation
	\begin{equation}\label{equationsearch}
		\left\{
		\begin{array}{ll}
			v_{tt}-\Delta v + v + (v + F_{\geq N}^\omega)^\frac{d+2}{d-2} = 0, \quad (t,x)\in \mathbb{R}\times \mathbb{R}^d;\\
			v(0) = f^\omega_{<N}, \quad v_t(0) = g^\omega_{<N}.
		\end{array}
		\right.
	\end{equation}
	This equation is, essentially, a perturbation of the equation (\ref{aim0})
	 . The main term of this perturbation is $v^\frac{4}{d-2}F_{\geq N}^\omega$. Thus, the proof of Theorem \ref{finallll} is decomposed into two parts. 
	\begin{itemize}
		\item Suppose some conditions on $F (F_{\geq N}^\omega)$ such that we can obtain the local, global and scattered solution of the equation (\ref{equationsearch}).
		\item For some $s<1$, $N=N(\omega)<\infty, a.e.$,  we can show that $F_{\geq N}^\omega$ satisfies the conditions appeared in the first part.
	\end{itemize}

	For the first part, in subsection \ref{localexistenceeee}, \ref{globalexistenceeee}, we use perturbation argument to obtain the local and global existence by assuming some time-space integrability conditions on $F$. The argument is classical. To obtain the scattering result in subsection \ref{scatteringsection}, we assume some conditions originated in \cite{Bringmann1}. These conditions that appear unnaturally are intended for the induction on scales argument. 
	
	For the second part, we use probabilistic Strichartz estimates. The free evolution of the random data satisfies significantly improved time-space integrability properties. In subsection \ref{alomostsurelocalglobal}, we show the almost sure global existence for $s>0$. To obtain the scattering result in subsection \ref{Almostsurescattering}, we use the wave packet decomposition from \cite{Bringmann1}. We put the proof of the almost sure finiteness of the key wave packet estimates in Appendix.
	
	In subsection 2.1, we show some basic almost orthogonal estimates of phase space localization. Then, we recall the decay estimates and related Strichartz estimates related to the linear Klein-Gordon equation in subsection 2.2. In subsection 2.3, we recall a few basic estimates for sub-Gaussian random variables. In section 3, 4, we proceed as previously described.
	\vspace{5pt}
	
	\textbf{Some notations:} $\langle x \rangle:= (1+|x|^2)^\frac{1}{2}$, and similarly $\langle \nabla \rangle = \mathscr{F}^{-1}\langle \xi \rangle \mathscr{F}$.
	
	Let $\chi_0 \in C_0^\infty(\mathbb{R}^d)$ satisfying 
	$0\leq \chi_0\leq 1,~ \chi_0|_{B(0,1)}\equiv 1,~ \chi_0|_{B^c(0,2)}\equiv 0$.
	Then, define
	$\psi_1 = \chi_0, \quad \psi_N(\xi) = \chi_0(\xi/N)-\chi_0(2\xi/N),  ~\forall~N\in 2^{\mathbb{N}}$.
	Denote $P_N = \mathscr{F}^{-1}\psi_N \mathscr{F},~ N\in 2^{\mathbb{N}_0},~\mathbb{N}_0 = \mathbb{Z}^+\cup\{0\}$. 
	We also need the enlarged dyadic decomposition and uniform decomposition defined as 
	$$\tilde{P}_{1} := P_1+P_2, \quad \tilde{P}_N := P_{{N}/{2}}+P_N+P_{2N}, ~N\in 2^{\mathbb{N}},$$ 
	$$\tilde{P}_{k}:= \sum_{|\tilde{k}-k|_{\infty}\leq 1}P_{\tilde{k}}, ~k\in \mathbb{Z}^d.$$
	For $f\in \mathcal{S}'$, denote $f_{\geq N} = \sum_{M\geq N} P_Mf$, $f_{<N} = \sum_{M<N}P_M f$, and $f_N=P_N f$. We usually use $M, N$ to represent the dyadic numbers larger than or equal to $1$, and use $k,l$ to reprensent the elements in $\mathbb{Z}^d$.
	
	We denote 
	$$S(I):= L^\frac{d+2}{d-2}_t(I, L^\frac{2(d+2)}{d-2}_x).$$ 
	This norm is referred to as the Strichartz norm. $f\in L^q_{t,loc}(I,L_x^r)$ means $f\in L_t^q(J, L_x^r), ~\forall~J\subset\subset I$, similarly for $S_{loc}(I)$. For $u_0\in H^1, u_1\in L^2$, define
	\begin{align*}
		E(u_0,u_1) = \int_{\mathbb{R}^d} \frac{1}{2}u_1^2+\frac{1}{2}|\nabla u_0|^2+\frac{1}{2} u_0^2+\frac{d-2}{2d}u_0^\frac{2d}{d-2}dx.
	\end{align*}
	 	 
	 \section{Linear estimates and probabilistic inequalities}
	 \subsection{Almost orthogonal estimates}
	   	 \begin{lem}[Lemma 2.8 in \cite{Bringmann1}, Lemma 2.1 in \cite{Chen_2020}]\label{mismatch}
	 	Let $1\leq p,q\leq \infty$. Given $\phi, h\in \mathcal{S}$, denote $\phi_k(x) = \phi(x-k)$, $h_l(x) = h(x-l)$. For $k,k',l,l'\in \mathbb{Z}^d$, define
	 	\begin{align*}
	 		T_{k,l,l'} f &= h_l\mathscr{F}^{-1}(\phi_k\mathscr{F}(h_{l'}f)),& f\in \mathcal{S}',\\
	 		\tilde{T}_{k,k',l} f &= \mathscr{F}^{-1}(\phi_k\mathscr{F}(h_{l}\mathscr{F}^{-1} (\phi_{k'}\hat{f}))),&f\in \mathcal{S}'.
	 	\end{align*}
	 	Then, for any $L\in \mathbb{N}$, we have
	 	\begin{align}
	 		\|T_{k,l,l'}f\|_{L^q}&\lesssim_{\phi,h,L}\langle l-l'\rangle^{-L}\|f\|_{L^p}, \label{cutfirst}\\
	 		\|T_{k,k',l}f\|_{L^q}&\lesssim_{\phi,h,L}\langle k-k'\rangle^{-L}\|f\|_{L^p}. \label{fouriercutfirst}
	 	\end{align}
	 \end{lem}
	 \begin{proof}[\textbf{Proof}]
	 	$f(x)\rightsquigarrow f_{l'}(x)e^{ix\cdot k}$, we reduce the inequality (\ref{cutfirst}) to the case $k = 0$, $l' = 0$. 
	 	\begin{align*}
	 		T_{0,l,0}f(x) = h(x-l)\int_{\mathbb{R}^d}\check{\phi}(x-y)h(y)f(y)dy.
	 	\end{align*}
	 	By $h,\phi\in \mathcal{S}$, we have
	 	\begin{align*}
	 		|T_{0,l,0}f(x)| &\lesssim_M \int_{\mathbb{R}^d}\langle x-l\rangle^{-M}\langle x-y\rangle^{-M}\langle y \rangle^{-M}|f(y)|dy\\
	 		&\lesssim_M \langle l\rangle^{-L}\int_{\mathbb{R}^d}\langle x-l\rangle^{-M+L}\langle x-y\rangle^{-M+L}\langle y \rangle^{-M+L}|f(y)|dy.
	 	\end{align*}
	 	Choosing $M > L+d$, by Minkowski inequality and H\"{o}lder inequality, we obtain
	 	\begin{align*}
	 		\|T_{0,l,0}f\|_{L^q}\lesssim_M \langle l \rangle^{-L}\int_{\mathbb{R}^d}\langle y\rangle^{-M+L}|f(y)|dy\lesssim_{M} \langle l \rangle^{-L}\|f\|_{L^p}.
	 	\end{align*}
 		We obtain (\ref{cutfirst}).
 		
	 	Similarly, $f(x)\rightsquigarrow f_{l}(x)e^{ix\cdot k'}$, we reduce the inequality (\ref{fouriercutfirst}) to the case $l = 0$, $k' = 0$.
	 	\begin{align*}
	 		\tilde{T}_{k,0,0}f(x) = e^{ix\cdot k}\int_{\mathbb{R}^d}e^{-iy\cdot k}\check{\phi}(x-y)h(y)\int_{\mathbb{R}^d}\check{\phi}(y-z)f(z)dzdy.
	 	\end{align*}
	 	$(1-\Delta_y)^M e^{-iy\cdot k} = (1+|k|^2)^Me^{-iy\cdot k}$, integrating by parts, we obtain
	 	\begin{align*}
	 		|\tilde{T}_{k,0,0}f(x)|&\lesssim \langle k \rangle^{-2M}\int_{\mathbb{R}^d}|f(z)|dz\int_{\mathbb{R}^d}|(1-\Delta_y)^M (\check{\phi}(x-y)h(y)\check{\phi}(y-z))|dy\\
	 		&\lesssim_L \langle k\rangle^{-2M}\int_{\mathbb{R}^d}|f(z)|dz\int_{\mathbb{R}^d}\langle x-y\rangle^{-M}\langle y \rangle^{-M}\langle y-z\rangle^{-M}dy.
	 	\end{align*}
	 	Choosing $M$ such that $2M>L, M>d$, by Minkowski inequality, H\"{o}lder inequality, we obtain
	 	\begin{align*}
	 		\|\tilde{T}_{k,0,0}f\|_{L^q}&\lesssim \langle k \rangle^{-L} \int_{\mathbb{R}^d}\langle y \rangle^{-M} dy\int_{\mathbb{R}^d}\langle y-z\rangle^{-M}|f(z)|dz\\
	 		&\lesssim \langle k \rangle^{-L}\|f\|_{L^p}.
	 	\end{align*}
	 	Then, by translation and modulation, we have (\ref{fouriercutfirst}).
	 \end{proof}
	 \vspace{5pt}
	 By Lemma \ref{mismatch}, we can show the equivalent norm property of modulation spaces. See also Lemma 2.9 in \cite{Bringmann1}, Proposition 2.2 in \cite{Chen_2020} for similar results.
	 \begin{lem}\label{equivalentnorm}
	 	For $s\in \mathbb{R}$, $1\leq p,q,r\leq \infty$, we have
	 	\begin{equation}\label{equivalentnorkkk}
	 		\|\langle k \rangle^s\|P_k(\varphi_l f)\|_{L^r}\|_{l^q_k l^p_l}\sim \|\langle k\rangle^s \|P_k f\|_{L^p}\|_{l^q_k}.
	 	\end{equation}
	 	Specially, for $p = q = 2$, we have
	 	\begin{equation}\label{mainuseequivalentnorm}
	 		\|f\|_{H^s}\sim \|\langle k \rangle^s\|P_k f\|_{L^2}\|_{l^2}\sim \|\langle k \rangle^s\|P_k(\varphi_l f)\|_{L^r}\|_{l^2_k l^2_l}\sim \|\|\varphi_l f\|_{H^s}\|_{l^2_l} .
	 	\end{equation}
	 \end{lem}
	 \begin{rem}
	 	$\|\langle k\rangle^s \|P_k f\|_{L^p}\|_{l^q_k}$ is the equivalent norm of modulation space $M^s_{p,q}$ (see Feichtinger \cite{Feitingermodulation}). The proof of Lemma \ref{equivalentnorm} is essentially the same as Proposition 2.2 in \cite{Chen_2020} with some minor modifications.
	 \end{rem}
	 \begin{proof}[\textbf{Proof}] $\tilde{\varphi}_l := \sum_{|\tilde{l}-l|_\infty\leq 1}\varphi_{\tilde{l}}$. By $\varphi_l\tilde{\varphi}_l = \varphi_l, \sum_{l}\varphi_l = 1$, we have
	 	$$P_k(\varphi_l f) = \sum_{k',l'}P_k(\varphi_l {P}_{k'}({\varphi}_{l'}\tilde{\varphi}_{l'}\tilde{P}_{k'}f)).$$
	 	By Lemma \ref{mismatch}, we obtain
	 	\begin{align*}
	 		\langle k \rangle^s \|P_k(\varphi_l f)\|_{L^r}&\lesssim_L \sum_{k',l'}\langle k-k'\rangle^{-L}\langle k\rangle^s \langle l-l' \rangle^{-L}\|\tilde{\varphi}_{l'}\tilde{P}_{k'}f\|_{L^p}\\
	 		&\lesssim_L \sum_{k',l'}\langle k-k'\rangle^{-L+|s|} \langle l-l' \rangle^{-L}\langle k'\rangle^s\|\tilde{\varphi}_{l'}\tilde{P}_{k'}f\|_{L^p}.
	 	\end{align*}
	 	Choosing $L>|s|+d$, by Young inequality, we obtain
	 	\begin{align*}
	 		\|\langle k \rangle^s \|P_k(\varphi_l f)\|_{L^r}\|_{l^q_k l^p_l}\lesssim \|\langle k' \rangle^s\|\tilde{\varphi}_{l'}\tilde{P}_{k'}f\|_{L^p}\|_{l^q_{k'} l^p_{l'}}\lesssim \|\langle k \rangle^s \|P_k f\|_{L^p}\|_{l^q_k}.
	 	\end{align*}
	 	
	 	For another part, $\varphi_l P_k f = \sum_{k',l'} \varphi_l P_k({\varphi}_{l'}{P}_{k'}\tilde{P}_{k'}(\tilde{\varphi}_{l'}f))$, thus, we have
	 	\begin{align*}
	 		\|\varphi_l P_k f\|_{L^p}\lesssim \sum_{k',l'} \langle l-l'\rangle^{-L}\langle k-k'\rangle^{-L} \|\tilde{P}_{k'}\tilde{\varphi}_{l'}f\|_{L^r}.
	 	\end{align*}
	 	By Young inequality, $\|\varphi_l\|_{l^p_l}\lesssim 1$, for $L>|s|+d$,
	 	\begin{align*}
	 		\langle k \rangle^s\|P_k f\|_{L^p}&\lesssim \langle k\rangle^s\|\|\varphi_lP_k f\|_{L^p}\|_{l^p_l}\lesssim \sum_{k'}\langle k-k'\rangle^{-L}\langle k\rangle^s\|\|\tilde{P}_{k'}\tilde{\varphi}_{l'}f\|_{L^r}\|_{l^p_{l'}}\\
	 		&\lesssim\sum_{k'}\langle k-k'\rangle^{-L+|s|}\langle k'\rangle^s\|\|\tilde{P}_{k'}\tilde{\varphi}_{l'}f\|_{L^r}\|_{l^p_{l'}}.
	 	\end{align*}
	 	Then,
	 	\begin{align*}
	 		\|\langle k \rangle^s \|P_k f\|_{L^p}\|_{l^q_k}\lesssim\| \langle k'\rangle^s\|\|\tilde{P}_{k'}\tilde{\varphi}_{l'}f\|_{L^r}\|_{l^p_{l'}}\|_{l^q_{k'}}\lesssim \|\langle k\rangle^s\|P_k(\varphi_l f)\|_{L^r}\|_{l^q_k l^p_l}.
	 	\end{align*}
	 	
	 	For $p = q = 2$, by $M_{2,2}^s = H^s$, we have $\|f\|_{H^s}\sim \|f\|_{M_{2,2}^s}$. Then, for $r = 2$, we have $\|f\|_{H^s}\sim \|\langle k \rangle^s\|P_k(\varphi_l f)\|_{L^2}\|_{l^2_k l^2_l}\sim \|\|\varphi_l f\|_{H^s}\|_{l^2_l}$.
	 \end{proof}
	 \subsection{Linear estimates related to the linear Klein-Gordon equation}
	 Recall the decay estimates and
	  Strichartz estimates for the linear Klein-Gordon equation (see, e.g., \cite{Brenner}, \cite{GVdecay}).
	 \begin{lem}
	 		 $N\in 2^{\mathbb{N}_0}$, $f_N = P_N f, ~\forall~f\in \mathcal{S}'$, 
	 	\begin{equation*}
	 		\|e^{\pm it\langle \nabla \rangle}f_N\|_{L^\infty} \lesssim \min\{N^d,N^\frac{d+1}{2}|t|^{-\frac{d-1}{2}},N^\frac{d+2}{2}|t|^{-\frac{d}{2}} \}\|f_N\|_{L^1}.
	 	\end{equation*}
	 	The related Strichartz estimates, for $\frac{2}{q} + \frac{d}{r} \leq \frac{d}{2},~ q,r \geq 2, (q,r,d) \neq (2,\infty,2)$, $(q,r,d)\neq (2,\infty,3)$,
	 	\begin{equation*}
	 		\|e^{\pm it\langle \nabla \rangle}f_N\|_{L_t^qL_x^r}\lesssim \max\{N^{\frac{d}{2}-\frac{1}{q}-\frac{d}{r}},N^{\frac{1}{2}+\frac{1}{q}-\frac{1}{r}} \}\|f_N\|_{L^2}.
	 	\end{equation*}
 		For $(q,r,d) = (2,\infty,3)$, we have
 			$\|e^{\pm it\langle \nabla \rangle}f_N\|_{L_t^2L_x^\infty}\lesssim N\log^\frac{1}{2}(2N)\|f_N\|_{L^2}$.
	 \end{lem}
  
  	 
  	 \vspace{5pt}
  	 Similar to the wave equation case (Lemma 2.12 in \cite{Bringmann1}), we have better decay estimates when the frequence is restricted on a unit cube.
  	 \begin{lem}[Refined decay estimate]\label{refineddecayestimates} For $2\leq r\leq \infty$,
  	 	\begin{equation}\label{refineddecayestimate} 
  	 		\|e^{\pm it \langle \nabla \rangle} P_k f\|_{L^r}\lesssim \min\{1, \langle k\rangle^{\frac{d-1}{2}}|t|^{-\frac{d-1}{2}}, \langle k \rangle^{\frac{d+2}{2}}|t|^{-\frac{d}{2}}\}^{1-\frac{2}{r}}\|f\|_{L^{r'}}.
  	 	\end{equation}
  	 \end{lem}
   	 \begin{proof}[\textbf{Proof}] By interpolation with $r = 2$ which is easy to obtain by Plancherel indentity, we only need to show the case $r = \infty$. $e^{\pm it\langle\nabla\rangle}P_k f  = \mathscr{F}^{-1}(e^{\pm it\langle \cdot \rangle}\varphi(\cdot-k))*f$, by Young inequality, we have 
   	 	$$	\|e^{\pm it \langle \nabla \rangle} P_k f\|_{\infty}\lesssim \|\mathscr{F}^{-1}(e^{\pm it\langle \cdot \rangle}\varphi(\cdot-k))\|_{L^\infty}\|f\|_{L^1}.$$
   	 	Then,
   	 	\begin{align*}
   	 		\mathscr{F}^{-1}(e^{\pm it\langle \cdot \rangle}\varphi(\cdot-k))(x) &= \frac{1}{(2\pi)^d}\int_{\mathbb{R}^d}e^{\pm it \langle \xi\rangle}\varphi(\xi-k)e^{ix\cdot \xi}d\xi\\
   	 		& = \frac{e^{ix\cdot k}}{(2\pi)^d}\int_{\mathbb{R}^d}\varphi(\xi)e^{i(x\cdot \xi\pm t\langle\xi+k\rangle)}d\xi.
   	 	\end{align*}
    	$\Phi(\xi; t,k):= x\cdot \xi\pm t\langle \xi+k\rangle$, $\det(\nabla_\xi^2 \Phi(\xi;t,k)) = (\pm t)^d\langle \xi+k\rangle^{-(d+2)}$. $\varphi$ supports on $[-1,1]^d$. Thus, $|\det(\nabla_\xi^2 \Phi(\xi;t,k))|\gtrsim |t|^d\langle k\rangle^{-(d+2)}$, by nonstationary argument, we obtain the second part. By symmetry, we can assume $|k_d|\gtrsim |k|, k = (k',k_d)$. Then, $\det(\nabla_{\xi'}^2 \Phi(\xi;t,k)) = (\pm t)^{d-1}\langle\xi+k\rangle^{-d+1}\frac{\langle \xi_d+k_d\rangle^{2}}{\langle \xi+k\rangle^2}$, thus
    	$|\det(\nabla_{\xi'}^2 \Phi(\xi;t,k))|\gtrsim |t|^{d-1}\langle k \rangle^{-(d-1)}$. By nonstationary argument, we obtain
    	\begin{align*}
    		\|\mathscr{F}^{-1}(e^{\pm it\langle \cdot \rangle}\varphi(\cdot-k))\|_{L^\infty}\lesssim (\langle k\rangle ^{-1}|t|)^{-\frac{d-1}{2}}.
    	\end{align*}
		We also have $\|\mathscr{F}^{-1}(e^{\pm it\langle \cdot \rangle}\varphi(\cdot-k))\|_{L^\infty}\lesssim \|\varphi\|_{L^1}\lesssim 1$, then, (\ref{refineddecayestimate}).
   	 \end{proof}
  	 \vspace{5pt}
  	 It is standard to obtain Strichartz estimates by the decay estimates in Lemma \ref{refineddecayestimates}. For the endpoint case, see \cite{KeelTao}. We omit the proof.
 	 \begin{lem}[Refined Strichartz estimates]\label{Refinstr} 	
 	 	For $\frac{2}{q}+\frac{d-1}{r}\leq \frac{d-1}{2}, q,r\geq 2, (q,r,d)\neq (2,\infty,3)$, we have
 	 	\begin{equation*}
 	 		 \|e^{\pm it\langle \nabla \rangle} P_kf\|_{L^q_tL^r_x}\lesssim \langle k\rangle^\frac{1}{q} \|P_k f\|_{L^2}.
 	 	\end{equation*}
  	
  	    For 
  	    $d=3$, we have
  	    $\|e^{\pm it\langle \nabla \rangle} P_kf\|_{L^2_tL^\infty_x}\lesssim \langle k\rangle^\frac{1}{2}\log^\frac{1}{2} (2\langle k\rangle) \|P_k f\|_{L^2}$. 
  	    
  		For $\frac{2}{q}+\frac{d-1}{r}>\frac{d-1}{2}, \frac{2}{q}+\frac{d}{r}\leq \frac{d}{2}$, $q,r\geq 2$, $(q,r,d)\neq (2,\infty,2)$, we have
  		\begin{equation*}
  			\|e^{\pm it\langle \nabla \rangle} P_kf\|_{L^q_tL^r_x}\lesssim \langle k\rangle^{\frac{3}{q}-(d-1)(\frac{1}{2}-\frac{1}{r})} \|P_k f\|_{L^2}.
  		\end{equation*}
 	 \end{lem}
  	 \vspace{5pt}
  	 We often use the Lemma 2.6 in \cite{Bringmann1} to transform $L_t^\infty$ estimate to some $L^q_t, ~q<\infty$.
  	 \begin{lem}[\cite{Bringmann1}, Lemma 2.6]\label{inftytofine}
  	 	For $1\leq q,r\leq \infty$, $N\in 2^{\mathbb{N}_0}$, $f\in L^2$,
  	 	\begin{equation*}
  	 		\|e^{\pm it\langle \nabla \rangle} P_Nf\|_{L^\infty_t(\mathbb{R}, L^r_x)}\lesssim N^\frac{1}{q} \|e^{\pm it\langle\nabla\rangle}P_N f\|_{L^q_t(\mathbb{R}, L^r_x)}.
  	 	\end{equation*}
  	 \end{lem}
  
  	 \subsection{Probabilistic estimates}\label{randommmmm}
  	 We recall some basic estimates for sub-Gaussian random variables. See also the subsection 2.1 in \cite{Bringmann1}.
	 \begin{defi}
	 	Let $\Omega$ be a probability space. $X:\Omega\rightarrow \mathbb{R}$ is a random variable. If $\sup_{1\leq p<\infty}p^{-\frac{1}{2}}\|X\|_{L^p}<\infty$, we call $X$ sub-Gaussian, and define 
	 	\begin{equation}\label{subgaussiannn}
	 		\|X\|_{\Psi}:= \sup_{1\leq p<\infty}p^{-\frac{1}{2}}\|X\|_{L^p}.
	 	\end{equation}
	 \end{defi}
 	 \begin{rem}
 	 	There are other equivalent descriptions of the sub-Gaussian variables. $\|X\|_{\Psi}<\infty$  is equivalent to that for some $c>0$, 
 	 	$$|\{\omega:|X(\omega)|>\lambda\}|\leq 2e^{-c\lambda^2}, \quad \forall~\lambda>0.$$
 	 	If we assume that $X$ is mean zero and real valued, $\|X\|_{\Psi}<\infty$ is also equivalent to that for some $M>0$, 
 	 	$$\int_{\mathbb{R}}e^{\gamma x}d\mu(x)\leq e^{M\gamma^2}, \quad \forall~\gamma\in \mathbb{R},$$
 	 	where $\mu(x):= |\{\omega:X(\omega)<x\}|, ~x\in \mathbb{R}$.
 	 	In fact, the best $M, c$ satisfy $M\sim \|X\|_{\Psi}^2\sim \frac{1}{c}$. See for example, \cite{burq2008random1, Propbability}.
 	 \end{rem}
     
	 We mainly use the following two estimates.  For the proof, see for example \cite{Propbability}, \cite{Bringmann1}.
  	 \begin{lem}[Khinchin inequality]
  	 	Let $\{X_j\}_{j=1}^J$ be independent, zero-mean
  	 	 sub-Gaussian random variables. Then, for any given $\{a_j\}_{j=1}^J\subset \mathbb{C}$, $1\leq p<\infty$,
  	 	\begin{equation*}
  	 		\left\|\sum_{j=1}^J a_j X_j\right\|_{L^1}\leq C p^\frac{1}{2}\|a_{j}\|_{l^2_{1\leq j\leq J}}\max_{1\leq j\leq J}\|X_j\|_{\Psi}.
  	 	\end{equation*}
   		The constant $C$ does not rely on $p, J$.
  	 \end{lem}
     \begin{lem}[Lemma 2.4 in \cite{Bringmann1}]\label{supsubgaussianlemmma}
     	Assume that $\{X_j\}_{j=1}^J$ are sub-Gaussian random variables. Then,
     	\begin{equation*}
     		\left\|\max_{1\leq j\leq J}|X_j|\right\|_{L^1}\leq C\log \langle J \rangle \max_{1\leq j\leq J} \|X_j\|_{\Psi}.
     	\end{equation*}
     	The constant $C$ does not rely on $J$.
     \end{lem}
	 
	 \section{Deterministic perturbation equations}
	 \subsection{Local existence}\label{localexistenceeee}
	 \begin{pro}\label{localexistence}
	 	Given
	 	$F\in S_{loc}(\mathbb{R}), (v_0,v_1)\in \mathcal{H}^1$,	the equation 
	 	\begin{equation}\label{perturb}
	 	\left\{
	 	\begin{array}{l}
	 	v_{tt} - \Delta v + v + (v+F)^\frac{d+2}{d-2} = 0, \quad (t,x)\in \mathbb{R}\times \mathbb{R}^d;\\
	 	v|_{t=0} = v_0 , ~ v_t|_{t = 0} = v_1 .
	 	\end{array}
	 	\right.
	 	\end{equation}
	 	has a unique solution
	 	$(v(t),v_t(t))\in C(I^*,\mathcal{H}^1),~v\in S_{loc}(I^*)$. $I^*$ is an open interval which contains $0$. The solution satisfies the integral equation, $\forall~t\in I^*$,
	 	\begin{equation*}
	 	v(t) = \cos(t\langle \nabla \rangle)v_0+\frac{\sin t\langle \nabla \rangle}{\langle \nabla \rangle}v_1 - \int_{0}^{t}\frac{\sin(t-s)\langle \nabla \rangle}{\langle \nabla \rangle}(v+F)^\frac{d+2}{d-2}(s)ds.
	 	\end{equation*}
	 	If $I^*\neq \mathbb{R}$, then $\|v\|_{S(I^*)} = \infty$. If $\|v\|_{S(I^*)}<\infty$, then $v$ scatters. It means that there exists $(v_0^\infty,v_1^\infty)\in \mathcal{H}^1$ such that
	 	\begin{equation*}
	 	\lim_{t\rightarrow \infty}\|(v(t),v_t(t))-K(t)(v_0^\infty,v_1^\infty)\|_{\mathcal{H}^1} = 0.
	 	\end{equation*}
	 	The similar statement also holds as $t\rightarrow -\infty$.
	 \end{pro}
 	 \begin{proof}[\textbf{Proof}] 
	 Constructing contraction map in $S(I)$,
	 for $v\in X(I)$, define
	 \begin{equation*}
	 	\mathcal{T}:v \mapsto \cos(t\langle \nabla \rangle)v_0+\frac{\sin t\langle \nabla \rangle}{\langle \nabla \rangle}v_1 - \int_{0}^{t}\frac{\sin(t-s)\langle \nabla \rangle}{\langle \nabla \rangle}(v+F)^\frac{d+2}{d-2}(s)ds.
	 \end{equation*} 
 	Then, by Strichartz estimates,
	 \begin{equation*}
	 	\|\mathcal{T}v\|_{S(I)}\leq \left\|\cos(t\langle \nabla \rangle)v_0+\frac{\sin t\langle \nabla \rangle}{\langle \nabla \rangle}v_1\right\|_{S(I)}+C\|v\|_{S(I)}^\frac{d+2}{d-2}+C\|F\|_{S(I)}^\frac{d+2}{d-2}.
	 \end{equation*}
	 \begin{equation*}
	 	\|\mathcal{T}v_1-\mathcal{T}v_2\|_{S(I)}\leq C\|v_1-v_2\|_{S(I)}(\|v_1\|_{S(I)}^\frac{4}{d-2}+\|v_2\|_{S(I)}^\frac{4}{d-2}+\|F\|_{S(I)}^\frac{4}{d-2}).
	 \end{equation*}
	 \begin{equation*}
	 	\mathcal{D} = \{v\in X(I):\|v\|_{S(I)}\leq \eta \}.
	 \end{equation*}
	 Choose $\eta>0$, such that
	 \begin{equation*}
	 	C\eta ^\frac{4}{d-2}\leq \frac{1}{2},\quad C\big(2\eta^\frac{4}{d-2}+\|F\|^\frac{4}{d-2}_{S(I)}\big)\leq \frac{1}{2}.
	 \end{equation*}
	 Also choose $I$ sufficient small, such that
	 \begin{equation}
	 	\left\|\cos(t\langle \nabla \rangle)v_0+\frac{\sin t\langle \nabla \rangle}{\langle \nabla \rangle}v_1\right\|_{S(I)}+C\|F\|_{S(I)}^\frac{d+2}{d-2}\leq \frac{1}{2}\eta.
	 \end{equation}
	 By contraction mapping principle, we obtain unique fixed point $v\in X(I)$. Also, 
	 \begin{equation*}
	 	\|(v(t),v_t(t))\|_{C(I,\mathcal{H}^1)}\lesssim \|(v_0,v_1)\|_{\mathcal{H}^1}+\|v\|_{S(I)}^\frac{d+2}{d-2}+\|F\|_{S(I)}^\frac{d+2}{d-2}.
	 \end{equation*}
	 From above proof of local existence, we can obtain the following finite time blow up criterion. If $I^* \neq \mathbb{R}$, we have $\|v\|_{S(I^*)} = \infty$.
	\end{proof}
	 
	 \subsection{Global existence}\label{globalexistenceeee}
	 	 For the equation $(\ref{perturb})$, define the energy of $(v,v_t)\in \mathcal{H}^1$ by
	 \begin{equation*}
	 e(t) := E(v(t),v_t(t)) 
	 \end{equation*}
	 \begin{pro}\label{global existence}
	 	Suppose the conditions in Proposition \ref{localexistence}, and assume that $F\in L^1_{t,loc}(\mathbb{R}, L^\frac{2d}{d-4}_x)$, then the solution constructed in Proposition \ref{localexistence} is global. It means $I^* = \mathbb{R}$.
	 \end{pro}
  
  	 \begin{proof}[\textbf{Proof}] First, we show that $e(t)$ is bounded in any interval $I\subset \subset I^*$.
  	 	\begin{equation*}
  	 	e'(t) = \int_{\mathbb{R}^d}v_t (v^\frac{d+2}{d-2}-(v+F)^\frac{d+2}{d-2})dx.
  	 	\end{equation*}
  	 	By H\"{o}lder inequality,
  	 	\begin{align*}
  	 		|e'(t)|&\leq C_d\int_{\mathbb{R}^d}|v_t(t)||F(t)|(|v(t)|^\frac{4}{d-2}+|F(t)|^\frac{4}{d-2})dx\\
  	 		&\lesssim_d \|v_t(t)\|_{L^2}\|v(t)\|_{L^\frac{2d}{d-2}}^\frac{4}{d-2}\|F(t)\|_{L^\frac{2d}{d-4}}+\|v_t(t)\|_{L^2}\|F(t)\|_{L^\frac{2(d+2)}{d-2}}^\frac{d+2}{d-2}\\
  	 		&\lesssim_d e(t)^\frac{d+4}{2d}\|F(t)\|_{L^\frac{2d}{d-4}}+e(t)^\frac{1}{2}\|F(t)\|_{L^\frac{2(d+2)}{d-2}}^\frac{d+2}{d-2}.
  	 	\end{align*}
  	 	Thus, we obtain, for $t\geq 0$,
  	 	$$e(t)\leq e(0)+C_d\int_0^t e(s)^\frac{d+4}{2d}\|F(s)\|_{L^\frac{2d}{d-4}}+e(s)^\frac{1}{2}\|F(s)\|_{L^\frac{2(d+2)}{d-2}}^\frac{d+2}{d-2}ds.$$  
  	 	Define	$\tilde{e}(t) = \sup\limits_{0\leq s\leq t}e(s)$.
  	 	
  	 	For $d =4$,
  	 	\begin{align*}
  	 		\tilde{e}(t)&\leq e(0)+C_d\int_0^t \tilde{e}(s)\|F(s)\|_{L^\infty}ds+ C_d\tilde{e}(t)^\frac{1}{2}\|F\|_{S([0,t])}^3\\
  	 		&\leq e(0)+C_d\int_0^t \tilde{e}(s)\|F(s)\|_{L^\infty}ds+ \frac{1}{2}\tilde{e}(t)+4C_d^2\|F\|_{S([0,t])}^6.
  	 	\end{align*} 
  	 	Thus, $\tilde{e}(t)\lesssim e(0)+\|F\|_{S([0,t])}^6+\int_0^t \tilde{e}(s)\|F(s)\|_{L^\infty}ds$. 
  	 	By Gronwall inequality, we obtain
  	 	\begin{equation}\label{en4}
  	 		e(t)\leq \tilde{e}(t) \lesssim \left(e(0) + \|F\|_{S([0,t])}^6\right)\left(1+\|F\|_{L^1([0,t],L^\infty)}e^{\|F\|_{L^1([0,t],L^\infty)}}\right).
  	 	\end{equation} 
  	 	
  	 	For $d= 5$, 
  	 	\begin{align*}
  	 	\tilde{e}(t)&\leq e(0)+ C_d\tilde{e}(t)^\frac{9}{10}\|F\|_{L^1([0,t], L^{10})}+C_d\tilde{e}(t)^\frac{1}{2}\|F\|_{S([0,t])}^\frac{7}{3}\\
  	 	&\leq e(0) + \frac{1}{4}\tilde{e}(t) + \tilde{C}_d\|F\|_{L^1([0,t], L^{10})}^{10}+\frac{1}{4}\tilde{e}(t)+ \tilde{C}_d\|F\|_{S([0,t])}^\frac{14}{3},
  	 	\end{align*}
  	 	We obtain
  	 	\begin{equation}\label{en5}
  	 	e(t)\leq \tilde{e}(t)\lesssim e(0) + \|F\|_{S([0,t])}^\frac{14}{3}+\|F\|_{L^1([0,t],L^{10})}^{10} .
  	 	\end{equation}
  	 	From (\ref{en4}), (\ref{en5}), we know that the energy of $v$ is bounded in any finite time interval if $F\in S_{loc}(\mathbb{R})\cap L^1_{t,loc}(\mathbb{R}, L^\frac{2d}{d-4})$.
  	 	
  	 	Assume $I^* = (-T_*, T^*)$, $T^*\neq \infty$, thus we have $e(t)\leq M, ~\forall ~t\in [0,T^*)$, where $M$ relies on $e(0), \|F\|_{S([0,T^*])}, \|F\|_{L^1([0,T^*], L^\frac{2d}{d-4})}$. By Theorem \ref{Nakanishi}, for $t_0\geq 0$, we have a global solution  $u$ of the equation
  	 	\begin{equation*}
  	 		\left\{
  	 		\begin{array}{l}
  	 			u_{tt}-\Delta u+u+u^\frac{d+2}{d-2} = 0, \quad (t,x)\in \mathbb{R}\times \mathbb{R}^d\\
  	 		u\big|_{t=t_0} = v(t_0), \quad u_t\big|_{t = t_0} = v_t(t_0).
  	 		\end{array}
  	 		\right.
  	 	\end{equation*}
  	 	And $\|u\|_{S(\mathbb{R})}\leq C(M), ~\forall ~0\leq t_0<T^*$. Fixed some sufficient small $\eta >0$ which will be determined by $C(M)$, choose $t_0$ such that $\|F\|_{S([t_0,T^*])}\leq \eta$. Consider the equation for $w = v-u$.
  	 	\begin{equation*}
  	 	\left\{
  	 	\begin{array}{l}
  	 	w_{tt}-\Delta w+w = u^\frac{d+2}{d-2}-(u+F+w)^\frac{d+2}{d-2}, \quad (t,x)\in [t_0,T^*)\times \mathbb{R}^d\\
  	 	w\big|_{t=t_0} = 0, \quad w_t\big|_{t = t_0} = 0.
  	 	\end{array}
  	 	\right.
  	 	\end{equation*}
  	 	Then, we have
  	 	\begin{align*}
  	 		w(t) = \int_{t_0}^t \frac{\sin(t-s)\langle \nabla \rangle}{\langle \nabla \rangle}(u^\frac{d+2}{d-2}-(u+F+w)^\frac{d+2}{d-2})(s)ds.
  	 	\end{align*}
  	 	Choose $\delta>0$ sufficient small which is determined by $d$ . 
  	 	 Decompose $[t_0, T^*)$ into ${N}$ intevals $I_{k}:=[t_{k},t_{k+1})$, $t_0<t_1<\cdots<t_{N} < T^*$, such that $\|u
  	 	\|_{S(I_{k})}\leq \delta, ~k=0,1,\cdots,{N}-1$, and $
  	 	{N}\leq 2 \left(\frac{C(M)}{\delta}\right)^\frac{d+2}{d-2}\lesssim C(M)^\frac{d+2}{d-2}$.
  	 	(We may assume $C(M)\geq 1$). 	For $t\in I_{k}$, 
  	 	\begin{align*}
  	 	w(t) &=\pi_1 K(t-t_{k})(w(t_{k}),w_t(t_{k})) \\
  	 	&\quad-\int_{t_{k}}^{t}\frac{\sin(t-s)\langle \nabla \rangle}{\langle \nabla \rangle}\left((w+u+F)^\frac{d+2}{d-2}(s)-u^\frac{d+2}{d-2}(s)\right)ds.
  	 	\end{align*}
  	 	Then, by Strichartz estimates,
  	 	\begin{align*}
  	 	&\quad~ \|(w(t_{k+1}),w_t(t_{k+1}))\|_{\mathcal{H}^1}+\|w\|_{S(I_{k})}\\
  	 	&\lesssim \|(w(t_{k}),w_t(t_{k}))\|_{\mathcal{H}^1}+ \|w\|_{S(I_k)}^\frac{d+2}{d-2}+\|F\|_{S(I_k)}^\frac{d+2}{d-2}\\
  	 	&\qquad\qquad\qquad\qquad\qquad+(\|w\|_{S(I_{k})}+\|F\|_{S(I_{k})})\|u\|_{S(I_{k})}^\frac{4}{d-2}.
  	 	\end{align*}
  	 	Define
  	 	\begin{equation*}
  	 	A_k :=\|w\|_{S(I_{k})},\quad B_k = \|(w(t_{k}),w_t(t_{k}))\|_{\mathcal{H}^1}.
  	 	\end{equation*}
  	 	Thus, we have
  	 	\begin{equation*}
  	 	A_k\leq CB_k+CA_k^\frac{d+2}{d-2}+C\eta^\frac{d+2}{d-2}+C\delta^\frac{4}{d-2}(A_k+\eta).
  	 	\end{equation*}
  	 	also,
  	 	\begin{equation*}
  	 	B_{k+1}\leq CB_k+CA_k^\frac{d+2}{d-2}+C\eta^\frac{d+2}{d-2}+C\delta^\frac{4}{d-2}(A_k+\eta).
  	 	\end{equation*}
  	 	Choose $\delta$ such that  $C\delta^\frac{4}{d-2} = \frac{1}{2}$. Also $\eta$ small satisfies $2C\eta^\frac{4}{d-2}\leq 1$, 
  	 	then
  	 	\begin{align*}
  	 	A_k&\leq 2CB_k+2CA_k^\frac{d+2}{d-2}+2C\eta^\frac{d+2}{d-2}+2C\delta^\frac{4}{d-2}\eta\\
  	 	&\leq 2CB_k+2CA_k^\frac{d+2}{d-2}+2\eta;\\
  	 	B_{k+1}&\leq CB_k+CA_k^\frac{d+2}{d-2}+C\eta^\frac{d+2}{d-2}+C\delta^\frac{4}{d-2}(A_k+\eta)\\
  	 	&\leq CB_k + CA_k^\frac{d+2}{d-2}+\frac{1}{2}A_k+\eta.
  	 	\end{align*}
  	 	By bootstrap argument, assume
  	 	\begin{equation}\label{bootstrap}
  	 	2C\max_{k=0,1,\cdots,{N}-1}A_k^\frac{4}{d-2}\leq \frac{1}{2}.
  	 	\end{equation}
  	 	We have
  	 	\begin{equation*}
  	 	A_k\leq 4CB_k + 4\eta< 4CB_k + 8\eta,
  	 	\end{equation*}
  	 	\begin{equation*}
  	 	B_{k+1}\leq CB_k + \frac{3}{4}A_k+\eta\leq 4CB_k+4\eta.
  	 	\end{equation*}
  	 	$B_0 = 0$, we obtain
  	 	\begin{align*}
  	 	B_k \leq 4\eta \frac{(4C)^{{N}}-1}{4C-1},\quad k=1,\cdots,{N};\\
  	 	A_k< 4CB_k+8\eta \leq 4\eta \frac{(4C)^{{N}}+4C-2}{4C-1}, \quad k=0,1,\cdots,{N}-1.
  	 	\end{align*}
  	 	From the bootstrap condition (\ref{bootstrap}), we need $\eta$ such that
  	 	\begin{equation*}
  	 	2C\cdot \left(4\eta \frac{(4C)^{{N}}+4C-2}{4C-1}\right)^\frac{4}{d-2}\leq \frac{1}{2}.
  	 	\end{equation*}
  	 	Since $N\lesssim C(M)^\frac{d+2}{d-2}$, it is enough to choose 
  	 	\begin{equation}\label{chooseeta}
  	 			\eta = {(C_0)^{-1}}{\exp({-C_0C(M)^\frac{d+2}{d-2}})}
  	 	\end{equation}
  	 	for some $C_0(C,d)$. For the validity of the bootstrap argument, we use the ``continuity" of $A_k$. To see this, we show the validity of the argument for $A_0$. $A_0 = \|v\|_{S(I_0)}, I_0 = [t_0,t_1)$. Since $f(t):=\|v\|_{S([t_0,t))}$ is continuous function of $t$. $f(t_0) = 0$. Thus, $2Cf(t_0) = 0\leq \frac{1}{2}$. By the former argument, we show that $2Cf(t)\leq \frac{1}{2}$, then $2Cf(t)<\frac{1}{2}-8C\eta$ for $t<t_1$. Thus, $A_0 = f(t_1)$ satisfies $2CA_0\leq \frac{1}{2}$. Similar argument works for other $k$.
  	 	
  	 	Thus, we obtain
  	 	\begin{equation*}
  	 	\|w\|_{S([t_0,t_N))}\leq \left(\sum_{k=0}^{{N}-1}A_k^\frac{d+2}{d-2}\right)^\frac{d-2}{d+2}\leq {N}^\frac{d-2}{d+2}\lesssim C(M).
  	 	\end{equation*}
  	 	Since the estimate is uniform for $t_N<T^*$, we have $\|w\|_{S([t_0,T^*))}\lesssim C(M)$. $\|v\|_{S([t_0,T^*))}\leq \|w\|_{S([t_0,T^*))}+\|u\|_{S([t_0,T^*))}\lesssim C(M)<\infty$. From the blow up criterion, we obtain a contradiction. Thus, $T^* = \infty$, $v$ is global.
  	 \end{proof}
	 	 
	 \subsection{Uniform bound of energy implies scattering}	 
	 We use finite times the argument of the proof of Proposition \ref{global existence} to obtain: 
	 
	 \begin{pro}\label{boundenergyimplyscatter}
	 	Given $M_0:=\|F\|_{S(\mathbb{R})}<\infty$.  Assume that the solution of (\ref{perturb}) $v$ satisfies $$M:=\sup\limits_{t\in I^*}e(t)<\infty$$
	 	Then $v$ is global, and $v\in S(\mathbb{R})$. $C(M)$ is the quantity from Theorem \ref{Nakanishi}. We have
	 	\begin{equation*}
	 		\|v\|_{S(\mathbb{R})}\leq C_d M_0 (C(M)+1)e^{C_d (C(M)+1)^\frac{d+2}{d-2}}.
	 	\end{equation*}
	 	Thus, $v$ scatters. 
	 \end{pro}

	 The related argument was obtained in \cite{DLM} for wave equation. The argument here is almost the same.
	 
	 \begin{proof}[\textbf{Proof}]
	 Decompose $I^*$ to $\tilde{N}$ intevals $I_j, j=1,2,\cdots, \tilde{N}$, such that
	 	$\|F\|_{S(I_j)} \leq \eta$, $\tilde{N} \lesssim \left(\frac{M_0}{\eta}\right)^\frac{d+2}{d-2}$ where the $\eta$ is choosen as in the proof of Proposition \ref{global existence} by (\ref{chooseeta}). From the proof of Proposition \ref{global existence}, we have $\|v\|_{S(I_j)}\lesssim C(M)$. Thus,
	 \begin{align*}
	 	\|v\|_{S(I^*)}&\leq \left(\sum_{k=1}^{{\tilde{N}}}\|v\|_{S(I_j)}^\frac{d+2}{d-2}\right)^\frac{d-2}{d+2}\lesssim C(M)\tilde{N}^\frac{d-2}{d+2}\\
	 	&\lesssim M_0C(M)\exp(C_0C(M)^\frac{d+2}{d-2}).
	 \end{align*}
	We obtain that $v$ scatters by Proposition \ref{localexistence}.
	\end{proof}
	 
	 \subsection{Uniform energy bound by induction on scales}\label{scatteringsection}
	 To describe the conditions on $F$ to obtain the uniform energy bound as was shown in section 5, 6 in \cite{Bringmann1}, we need some notations for the decomposition of $\mathbb{R}_t\times \mathbb{R}^d_x$, the local energy and the local nonlinear force term.
	 $$K_{t_0,x_0}^N:= \{(t,x):t_0\leq t\leq t_0+N,|x-x_0|_\infty\leq 2N-t+t_0\},$$
	 $$K_{t_0,x_0}^N(t):=\{x:(t,x)\in K_{t_0,x_0}^N\}.$$
	 \begin{equation*}
	 	\tilde{K}_{t_0,x_0}^N:= \{(t,x):t_0\leq t\leq t_0+N,|x-x_0|_\infty\leq 10N-t+t_0\},
	 \end{equation*}
	 \begin{equation*}
	 	\tilde{K}_{t_0,x_0}^N(t):=\{x:(t,x)\in \tilde{K}_{t_0,x_0}^N\}.
	 \end{equation*}
	 For $u\in C(\mathbb{R}, H^1)\cap C^1(\mathbb{R}, L^2)$, $0<\delta<\frac{1}{50}$ define
	 \begin{equation*}
	 \tilde{\mathcal{E}}_{t_0,x_0}^N[u] = \sup_{t_0\leq t \leq t_0+N}\int_{ \tilde{K}_{t_0,x_0}^N(t)}\frac{1}{2}u_t^2+\frac{1}{2}u^2+\frac{1}{2}|\nabla u|^2+\frac{d-2}{2d}u^\frac{2d}{d-2}, 
	 \end{equation*}
	 \begin{equation*}
	 \tilde{\mathcal{F}}_{t_0,x_0}^{N,\delta}[u] = \sup_{\substack{t':t_0\leq t'\leq t_0+N\\x':|x'-x_0|\leq 3N}}\int_{t_0}^{t_0+N}\int_{\substack{\left||x-x'|-|t-t'|\right|\leq N^{10\delta}}}u(t,x)^{\frac{2d}{d-2}}dxdt.
	 \end{equation*} 
	 Define $\|f\|_{\tilde{S}(U)}:=\|f\chi_U\|_{S(\mathbb{R})}$ for $U$ bounded region in $\mathbb{R}\times \mathbb{R}^d$.
	 
	 \begin{pro}\label{uniformenergybyinduction}
	 	For $d = 4,5$.
	 	Given $\delta > 0, \theta> 0, \alpha > \theta+20\delta, \beta > 0$, there exists $\eta>0$ relies on $(\alpha,\beta,\theta,\delta)$, if $F = \sum\limits_{N\in 2^{\mathbb{N}_0}} F_N$, such that
	 	\begin{itemize}
	 		\item [$(i)$.] $F\in S(\mathbb{R})$,
	 		\item [$(ii)$.] $\|F\|_{L_t^1(I, L_x^\frac{2d}{d-4})}\leq \eta, ~\forall$ interval $I\subset \mathbb{R}$ with length $1$, 
	 		\item [$(iii)$.] $\|F_N\|_{L_t^1([N^{1+\theta},\infty),L_x^\infty)}\leq \eta N^{-\beta}$,
	 		\item [$(iv)$.] $t_0 = 0, N, 2N,\cdots,\lfloor N^\theta \rfloor N$, $x_0\in N\mathbb{Z}^d$.
	 		\begin{equation}
	 		\int_{K_{t_0,x_0}^N} |F_M ||u_t||u|^{\frac{4}{d-2}} \leq \eta M^{-\alpha}(\tilde{\mathcal{E}}_{t_0,x_0}^N[u]+\tilde{\mathcal{F}}_{t_0,x_0}^{N,\delta}[u]),
	 		\end{equation}
	 		for all dyadic $M\geq N$, $\forall~u\in C(\mathbb{R},{H}^1)\cap C^1(\mathbb{R},L^2)$.
	 	\end{itemize}	
 		Then, the solution $v$ of equation (\ref{perturb}) is global and scatters.
	 \end{pro}
 
 	\begin{rem}
 		The essential part of the proof Proposition \ref{uniformenergybyinduction} is the induction on scales argument from \cite{Bringmann1}. We need not to assume that  $F_N$ is $P_N F$ here. When we apply this proposition to (\ref{equationsearch}), $F_N$ is in fact the part of dyadic frequence decomposition of $F$.
 	\end{rem}

	 \begin{proof}[\textbf{Proof}] 
	 Due to the conditions (i), (ii) of $F$, we konw that the solution of (\ref{perturb}) $v$ is global by Proposition \ref{global existence}. By Proposition \ref{uniformenergybyinduction}, we only need to show that $E(v(t),v_t(t))$ is uniformly bounded. Consider the localized equation of $v$,
	 \begin{equation}\label{local}
	 	\left\{
	 	\begin{array}{l}
	 		w_{tt}-\Delta w + w + (w+F\chi_{K_{t_0,x_0}^N})^\frac{d+2}{d-2}=0,\quad (t,x)\in \mathbb{R}\times \mathbb{R}^d;\\
	 		(w(t_0),w_t(t_0))|_{\substack{|x-x_0|_{\infty}\leq 2N}} = (v(t_0),v_t(t_0))|_{|x-x_0|_\infty\leq 2N}.
	 	\end{array}
	 	\right.
	 \end{equation}
	 We extend $(w(t_0),w_t(t_0))$ to be a function in $\mathcal{H}^1$, such that
	 \begin{align*}
	 	E(w(t_0),w_t(t_0))&\leq C_0\int_{|x-x_0|_\infty\leq 3N}v_t(t_0)^2+v(t_0)^2+|\nabla v(t_0)|^2 + v(t_0)^\frac{2d}{d-2}\\
	 	&:=C_0e_{t_0,x_0}^N.
	 \end{align*}
	 The constant $C_0$ depends on dimension $d$ only.
	 For the solution of $(\ref{local})$, denoted by $w_{t_0,x_0}^N\in C(\mathbb{R}, H^1)\cap C^1(\mathbb{R}, L^2)$, define
	 $$\tilde{e}_{t_0,x_0}^N := \mathcal{\tilde{E}}_{t_0,x_0}^N[w_{t_0,x_0}^N],\quad \tilde{f}_{t_0,x_0}^N:=\mathcal{\tilde{F}}_{t_0,x_0}^{N,\delta}[w_{t_0,x_0}^N].$$	 
	 From finite speed of propagation, we have
	 $w_{t_0,x_0}^N|_{K_{t_0,x_0}^N} = v|_{K^N_{t_0,x_0}}$.
	 Using the stratage of Bringmann \cite{Bringmann1}, we prove for some $B>0$, relies on $d, C_0$, such that
	 \begin{equation}\label{ener}
	 	\tilde{e}_{t_0,x_0}^N \leq 2C_0e_{t_0,x_0}^N+B \|F\|^\frac{2(d+2)}{d-2}_{\tilde{S}(K_{t_0,x_0}^N)},
	 \end{equation}
	 \begin{equation}\label{forc}     	
	 	\tilde{f}_{t_0,x_0}^N\leq BN^{10\delta}\left(e_{t_0,x_0}^N+ \|F\|^\frac{2(d+2)}{d-2}_{\tilde{S}(K_{t_0,x_0}^N)}\right),
	 \end{equation}	 	 
	 uniformly for $t_0\in N\mathbb{N}_0,x_0\in N\mathbb{Z}^d,N\in 2^{\mathbb{N}_0}$. To show (\ref{ener}), (\ref{forc}). We start from two estimates obtained by energy method. 
	 \begin{lem}\label{estimateet0x0N}
	 	For some $C>1$ depend on dimension $d$ only, we have
	 	\begin{equation}\label{t0x0ne}
	 		\tilde{e}_{t_0,x_0}^N\leq \frac{8}{7}C_0e_{t_0,x_0}^N+C\|F\|^\frac{2(d+2)}{d-2}_{\tilde{S}(K_{t_0,x_0}^N)}+C\int_{K_{t_0,x_0}^N}|F||w_t||w|^{\frac{4}{d-2}}.
	 	\end{equation}
	 	\begin{equation}\label{t0x0nf}
	 		\tilde{f}_{t_0,x_0}^N \leq CN^{10\delta}\left(\tilde{e}_{t_0,x_0}^N+\|F\|^\frac{2(d+2)}{d-2}_{\tilde{S}(K_{t_0,x_0}^N)} + \int_{ K_{t_0,x_0}^N}|F||w_t||w|^\frac{4}{d-2}\right).
	 	\end{equation}
	 \end{lem}
 	 We prove these estimates at the end of this section. We continue to show (\ref{ener}), (\ref{forc}).
 	 To abbreviate the notations, we use $w$ to represent $w_{t_0,x_0}^N$ when there is no confusion.
 	 
 	 For $N = 1$, $\forall~t\in [t_0,t_0+1]$, $\|w(t)\|_{L^\frac{2d}{d-2}}\lesssim \|\nabla w(t)\|_{L^2}\lesssim (\tilde{e}^N_{t_0,x_0})^\frac{1}{2}$, and also $\|w(t)\|_{L^\frac{2d}{d-2}}\lesssim (\tilde{e}^N_{t_0,x_0})^\frac{d-2}{2d}$, thus, for $d = 4,5$, we have $\|w(t)\|_{L^\frac{2d}{d-2}}\lesssim (\tilde{e}_{t_0,x_0}^N)^\frac{d-2}{8}$. Then, we have
 	 \begin{align*}
 	 	&\quad\int_{K_{t_0,x_0}^N}|F||w_t||w|^\frac{4}{d-2}\\
 	 	&\leq \|w_t\|_{L^\infty L^2(K_{t_0,x_0}^N)}\|w\|_{L^\infty L^\frac{2d}{d-2}(K_{t_0,x_0}^N)}^\frac{4}{d-2}\|F\|_{L^1 L^\frac{2d}{d-4}(K_{t_0,x_0}^N)}\\
 	 	&\lesssim \eta\tilde{e}_{t_0,x_0}^N.
 	 \end{align*}
 	 For $\eta<\frac{1}{2}C_0, B>2C$, we obtain (\ref{ener}), (\ref{forc}) with $N=1$.
 	 
	 Now, assume for $M<N, M\in 2^{\mathbb{N}_0}$, (\ref{ener}), (\ref{forc}) is right. Then, for $N$, we need to estimate the last term in (\ref{t0x0ne}).
	 \begin{align*}
	 \int_{K_{t_0,x_0}^N}|F||w_t||w|^{\frac{4}{d-2}}
	 &\leq \sum_{M}\int_{ K_{t_0,x_0}^N}|F_M||w_t||w|^\frac{4}{d-2}.
	 \end{align*}
	 For $M\geq N$, we use the condition (iv) of $F$,
	 \begin{equation}\label{largesale}
	 \begin{array}{ll}
     \sum\limits_{M\geq N}\int_{ K_{t_0,x_0}^N}|F_M||w_t||w|^\frac{4}{d-2}
	 &\leq \sum\limits_{M\geq N}\eta
	 M^{-\alpha}(\tilde{e}_{t_0,x_0}^N+\tilde{f}_{t_0,x_0}^N)\\
	 &\leq C_{\alpha}\eta N^{-\alpha}(\tilde{e}_{t_0,x_0}^N+\tilde{f}_{t_0,x_0}^N)
	 \end{array}
	 \end{equation}
	 for $M<N$, we use the induction hypothesis,
	 \begin{align*}
	 	\int_{ K_{t_0,x_0}^N}&|F_M||w_t||w|^\frac{4}{d-2}\leq \sum_{\tau\in 0,M,\cdots, \lfloor{M^\theta}\rfloor M}\sum_{\substack{y\in M\mathbb{Z}^d\\K_{\tau,y}^M\subset K^N_{t_0,x_0}}}\int_{ K_{\tau,y}^M}|F_M||w_t||w|^\frac{4}{d-2}\\
	 	&\quad\quad\quad\qquad\qquad\qquad+\int_{ K_{t_0,x_0}^N\cap [M^{1+\theta},\infty)\times \mathbb{R}^d}|F_M||w_t||w|^\frac{4}{d-2}.
	 \end{align*}
	 Note that in $K^M_{\tau,y}$, $w_{t_0,x_0}^N=v=w_{\tau,y}^M$. For the first term, by (\ref{ener}), (\ref{forc}), the almost disjoint property of $K_{\tau,y}^M$, and Minkonwski inequality, we have
	 \begin{align*}
	 	&\quad\sum_{\tau\in 0,M,\cdots, \lfloor{M^\theta}\rfloor M}\sum_{\substack{y\in M\mathbb{Z}^d\\K_{\tau,y}^M\subset K^N_{t_0,x_0}}}\int_{ K_{\tau,y}^M}|F_M||w_t||w|^\frac{4}{d-2}\\
	 	&\leq \sum_{\tau\in 0,M,\cdots, \lfloor{M^\theta}\rfloor M}\sum_{\substack{y\in M\mathbb{Z}^d\\K_{\tau,y}^M\subset K^N_{t_0,x_0}}}\eta M^{-\alpha}(\tilde{e}^M_{\tau,y}+\tilde{f}^M_{\tau,y})\\
	 	&\lesssim_B  \sum_{\tau\in 0,M,\cdots, \lfloor{M^\theta}\rfloor M}\sum_{\substack{y\in M\mathbb{Z}^d\\K_{\tau,y}^M\subset K^N_{t_0,x_0}}}\eta M^{-\alpha+10\delta}({e}^M_{\tau,y}+\|F\|_{\tilde{S}(K_{\tau,y}^M)}^\frac{2(d+2)}{d-2})\\
	 	&\lesssim_B \eta M^{-\alpha+\theta+10\delta}({e}^M_{\tau,y}+\|F\|_{\tilde{S}(K_{t_0,x_0}^N)}^\frac{2(d+2)}{d-2}).
	 \end{align*}
	 For the second term, we use (iii),
	 \begin{align*}
	 	&\int_{ K_{t_0,x_0}^N\cap [M^{1+\theta},\infty)\times \mathbb{R}^d}|F_M||w_t||w|^\frac{4}{d-2}\\
	 	&\leq \|F_M\|_{L^1([M^{1+\theta},\infty),L^\infty)}\|w_t\|_{L^\infty L^2(K_{t_0,x_0}^N)}\|w\|^\frac{d-4}{2}_{{L^\infty L^2(K_{t_0,x_0}^N)}}\|w\|^{\frac{d(6-d)}{2(d-2)}}_{L^\infty L^\frac{2d}{d-2}(K_{t_0,x_0}^N)}\\
	 	&\leq \eta M^{-\beta}\tilde{e}_{t_0,x_0}^N.
	 \end{align*}
	 Thus, we obtain
	 \begin{equation}\label{smallscale}
	 	\sum_{M< N}\int_{ K_{t_0,x_0}^N}|F_M||w_t||w|^\frac{4}{d-2}\lesssim_B \eta \tilde{e}^N_{t_0,x_0}+\eta\|F\|^\frac{2(d+2)}{d-2}_{\tilde{S}(K_{t_0,x_0}^N)}.
	 \end{equation}
	 Combining (\ref{t0x0ne}), (\ref{largesale}), (\ref{smallscale}), we have
	 \begin{align*}
	 	\tilde{e}_{t_0,x_0}^N&\leq \frac{8}{7}C_0 e_{t_0,x_0}^N+C\|F\|^\frac{2(d+2)}{d-2}_{\tilde{S}(K_{t_0,x_0}^N)}+ C_\alpha \eta N^{-\alpha}(\tilde{e}_{t_0,x_0}^N+\tilde{f}_{t_0,x_0}^N)\\
	 	&\quad + C_B\eta (\tilde{e}^N_{t_0,x_0}+\|F\|^\frac{2(d+2)}{d-2}_{\tilde{S}(K_{t_0,x_0}^N)}).
	 \end{align*}	 
	 By (\ref{t0x0nf}), we also have
	 \begin{align*}
	 	\tilde{f}_{t_0,x_0}^N&\leq CN^{10\delta}\tilde{e}_{t_0,x_0}^N+CN^{10\delta}\|F\|^\frac{2(d+2)}{d-2}_{\tilde{S}(K_{t_0,x_0}^N)}+ C_\alpha \eta N^{-\alpha+10\delta}(\tilde{e}_{t_0,x_0}^N+\tilde{f}_{t_0,x_0}^N)\\
	 	&\quad + C_B\eta N^{10\delta} (\tilde{e}^N_{t_0,x_0}+\|F\|^\frac{2(d+2)}{d-2}_{\tilde{S}(K_{t_0,x_0}^N)}).
	 \end{align*}
	 Due to $-\alpha+10\delta <0$, choose $\eta$ such that
	$
	 	C_\alpha\eta \leq \frac{1}{8}, C_B\eta\leq \frac{1}{8}, 4CC_\alpha\eta\leq \frac{1}{8}$, we obtain
	 \begin{equation}\label{backtothis}
	 	\tilde{f}_{t_0,x_0}^N\leq 4CN^{10\delta}\tilde{e}_{t_0,x_0}^N+4CN^{10\delta}\|F\|^\frac{2(d+2)}{d-2}_{\tilde{S}(K_{t_0,x_0}^N)}.
	 \end{equation}
	 Then,
	 \begin{align*}
		\tilde{e}_{t_0,x_0}^N&\leq \frac{8}{7}C_0 e_{t_0,x_0}^N+2C\|F\|^\frac{2(d+2)}{d-2}_{\tilde{S}(K_{t_0,x_0}^N)}+\frac{1}{4}\tilde{e}_{t_0,x_0}^N+ C_\alpha \eta N^{-\alpha}\tilde{f}_{t_0,x_0}^N\\
		&\leq \frac{8}{7}C_0 e_{t_0,x_0}^N+2C\|F\|^\frac{2(d+2)}{d-2}_{\tilde{S}(K_{t_0,x_0}^N)}+\frac{3}{8}\tilde{e}_{t_0,x_0}^N+ \frac{1}{8} N^{-\alpha+10\delta}\|F\|^\frac{2(d+2)}{d-2}_{\tilde{S}(K_{t_0,x_0}^N)}.
	\end{align*}
	 Thus, we have
	 \begin{equation*}
	 	\tilde{e}_{t_0,x_0}^N \leq 2C_0 e_{t_0,x_0}^N+4C\|F\|^\frac{2(d+2)}{d-2}_{\tilde{S}(K_{t_0,x_0}^N)}.
	 \end{equation*}
	 By (\ref{backtothis}), we have
	 \begin{equation*}
	 	\tilde{f}_{t_0,x_0}^N\leq 8CC_0N^{10\delta}{e}_{t_0,x_0}^N+(16C^2+4C)N^{10\delta}\|F\|^\frac{2(d+2)}{d-2}_{\tilde{S}(K_{t_0,x_0}^N)}.
	 \end{equation*}
	 Taking $B = \max\{16C^2+4C, 8CC_0\}$, we obtain (\ref{ener}) and (\ref{forc}).
	 
	 Now, fixed $t_0 = 0, x_0 = 0$, for any $t\in [0,\infty)$, we have
	 \begin{align*}
	 	E(v(t),v_t(t))\leq \liminf_{N\rightarrow\infty} \tilde{e}_{0,0}^N\lesssim E(v(0),v_t(0))+ \|F\|_{S(\mathbb{R})}^\frac{2(d+2)}{d-2}.
	 \end{align*}
	 Thus, we obtain the uniform bound for $v$ in the energy space. By Proposition \ref{boundenergyimplyscatter}, we conclude that $v$ scatters.
	\end{proof}

	\begin{proof}[\textbf{Proof of Lemma \ref{estimateet0x0N}}]
		Fixed $t_0,x_0, N$, we use $w$ to represent $w_{t_0,x_0}^N$.
		\begin{align*}
			E(w(t), w_t(t)) = E(w(t_0),w_t(t_0))+\int_{t_0}^t\int_{\mathbb{R}^d} w_t(w^\frac{d+2}{d-2}-(w+F\chi_{K_{t_0,x_0}^N})^\frac{d+2}{d-2}).
		\end{align*}
		Thus, for any $t\in[t_0,t_0+N)$, we have
		\begin{align*}
		E(w(t),w_t(t))\leq E(w(t_0),w_t(t_0))+C\int_{K_{t_0,x_0}^N}|w_t||F|(|F|^\frac{4}{d-2}+|w|^\frac{4}{d-2}).
		\end{align*}
		By H\"{o}lder inequality and the definition of $\tilde{e}^N_{t_0,x_0}$, we obtain
		\begin{align*}
		\tilde{e}_{t_0,x_0}^N&\leq E(w(t_0),w_t(t_0))+C\int_{K_{t_0,x_0}^N}|w_t||F|^\frac{d+2}{d-2}+C\int_{K_{t_0,x_0}^N}|F||w_t||w|^{\frac{4}{d-2}}\\
		&\leq C_0 e_{t_0,x_0}^N + C\|w_t\|_{L^\infty L^2(K_{t_0,x_0}^N)}\|F\|_{\tilde{S}(K_{t_0,x_0}^N)}^\frac{d+2}{d-2}+C\int_{K_{t_0,x_0}^N}|F||w_t||w|^\frac{4}{d-2}\\
		&\leq C_0e_{t_0,x_0}^N + \frac{1}{8}\tilde{e}_{t_0,x_0}^N+C\|F\|_{\tilde{S}(K_{t_0,x_0}^N)}^\frac{2(d+2)}{d-2}+C\int_{K_{t_0,x_0}^N}|F||w_t||w|^\frac{4}{d-2} 
		\end{align*}
		The constant $C$ may be different line by line, but relies on $d$ only. Thus, we obtain (\ref{t0x0ne}).
		
		We still need to show (\ref{t0x0nf}), for any $t'\in [t_0,t_0+N], |x'-x_0|\leq 3N$, $0\leq s\leq N^{10
			\delta}$, $t\leq t'+s$,
		\begin{align*}
		&\frac{d}{dt}\int_{|x-x'|\leq s+t'-t}\frac{1}{2}w_t(t)^2+\frac{1}{2}w(t)^2+\frac{1}{2}|\nabla w(t)|^2 + \frac{d-2}{2d}w(t)^\frac{2d}{d-2}dx\\
		&\quad = -\int_{|x-x'|= s+t'-t}\frac{1}{2}w_t(t)^2+\frac{1}{2}w(t)^2+\frac{1}{2}|\nabla w(t)|^2 + \frac{d-2}{2d}w(t)^\frac{2d}{d-2}d\sigma(x)\\
		&\quad\quad+ \int_{|x-x'|=s+t'-t}w_t(t)\frac{\partial w}{\partial n}d\sigma(x)\\
		&\quad\quad+\int_{|x-x'|\leq s+t'-t}w_t(t)((w(t)+F(t)\chi_{K_{t_0,x_0}^N})^\frac{d+2}{d-2}-w(t)^\frac{d+2}{d-2})dx.\\
		&\quad\leq  -\frac{d-2}{2d}\int_{|x-x'|= s+t'-t}w(t)^\frac{2d}{d-2}d\sigma(x) \\
		&\quad\quad + \int_{|x-x'|\leq s+t'-t}w_t(t)((w(t)+F(t)\chi_{K_{t_0,x_0}^N})^\frac{d+2}{d-2}-w(t)^\frac{d+2}{d-2})dx.
		\end{align*}
		Thus, we have
		\begin{align*}
		&\quad \int_{t_0}^{t'+s}\int_{|x-x'|= s+t'-t}w(t)^\frac{2d}{d-2}d\sigma(x)dt\\
		& \leq \frac{d}{d-2}\int_{|x-x'|\leq s+t'-t_0}w_t(t_0)^2+w(t_0)^2+|\nabla w(t_0)|^2 + \frac{d-2}{d}w(t_0)^\frac{2d}{d-2}dx\\
		&~ + \frac{2d}{d-2}\int_{t_0}^{t_2}\int_{|x-x'|\leq s+t'-t}w_t(t)\left|(w(t)+F\chi_{K_{t_0,x_0}^N}(t))^\frac{d+2}{d-2}-w(t)^\frac{d+2}{d-2}\right|dxdt\\
		&\lesssim  \tilde{e}_{t_0,x_0}^N+ \int_{{K}_{t_0,x_0}^N} |w_t||F|(|w|^\frac{4}{d-2}+|F|^\frac{4}{d-2})\\
		&\lesssim  \tilde{e}^N_{t_0,x_0} + \|w_t\|_{L^\infty L^2({K}_{t_0,x_0}^N)}\|F\|^\frac{d+2}{d-2}_{\tilde{S}({K}_{t_0,x_0}^N)}+\int_{{K}^N_{t_0,x_0}}|w_t||F|^\frac{d+2}{d-2}\\
		&\lesssim \tilde{e}_{t_0,x_0}^N + \|F\|^\frac{2(d+2)}{d-2}_{\tilde{S}({K}_{t_0,x_0}^N)}+\int_{{K}^N_{t_0,x_0}}|F||w_t||w|^\frac{4}{d-2}.
		\end{align*}
		Similarly, for any $t'\in [t_0,t_0+N]$, $|x'-x_0|\leq 3N$, $0\leq s\leq N^{10\delta}$, $t'-s\geq t_0$, 
		we have
		\begin{align*}
		&\quad \int_{t'-s}^{t_0+N}\int_{|x-x'|= s+t-t'}w(t)^\frac{2d}{d-2}d\sigma(x)dt\\
		&\lesssim \tilde{e}_{t_0,x_0}^N + \|F\|_{\tilde{S}({K}_{t_0,x_0}^N)}^\frac{2(d+2)}{d-2} + \int_{K_{t_0,x_0}^N} |F||w_t||w|^\frac{4}{d-2}.
		\end{align*}
		Similar estimates are also valid for $s<0$.
		For some $t'\in [t_0,t_0+N], |x'-x_0|\leq 3N$, we have
		\begin{align*}
		\tilde{f}_{t_0,x_0}^N &= \int_{t_0}^{t_0+N}\int_{||x-x'|-|t-t'||\leq N^{10\delta}}w(t,x)^\frac{2d}{d-2}dxdt\\
		&\lesssim \int_{0}^{N^{10\delta}}\chi_{t'-s\leq t_0+N}ds\int_{t'-s}^{t_0+N}\int_{|x-x'| = s+t-t'}w(t,x)^\frac{2d}{d-2}d\sigma(x)dt\\
		&\quad+ \int_{0}^{N^{10\delta}}ds\int_{t_0}^{t'+s}\int_{|x-x'| = s+t'-t}w(t,x)^\frac{2d}{d-2}d\sigma(x)dt\\
		&\quad+ \int_{0}^{N^{10\delta}}\chi_{s+t'\leq t_0+N}ds\int_{s+t'}^{t_0+N}\int_{|x-x'| = -s+t-t'}w(t,x)^\frac{2d}{d-2}d\sigma(x)dt\\
		&\quad+ \int_{0}^{N^{10\delta}}\chi_{t_0\leq t'-s}ds\int_{t_0}^{t'-s}\int_{|x-x'| = -s+t'-t}w(t,x)^\frac{2d}{d-2}d\sigma(x)dt\\
		&\leq CN^{10\delta}\left(\tilde{e}_{t_0,x_0}^N+\|F\|^\frac{2(d+2)}{d-2}_{\tilde{S}(K_{t_0,x_0}^N)} + \int_{ K_{t_0,x_0}^N}|F||w_t||w|^\frac{4}{d-2}\right).
		\end{align*}
		Thus, we obtain (\ref{t0x0nf}).
	\end{proof}
	 
	 \section{Probabilistic local, gloabl existence and scattering}
	 Denote $K_N:= \{k\in \mathbb{Z}^d:|k|_{\infty}\in (N,2N)\}, N\geq 2$, $K_1:=\{k\in \mathbb{Z}^d:|k|_{\infty}\leq 2\}$.
	 \subsection{Almost sure local and global existence}\label{alomostsurelocalglobal}
	 \begin{pro}\label{almostsureexistence}
	 	For $d = 4, 5$, $s>0$. For $a.e$ $\omega$, the solution of (\ref{aim}) is global.
	 \end{pro}
	 \begin{proof}[\textbf{Proof}]	 
	 Denote $
	 	f_N^\omega = \sum_{k\in K_N,l\in \mathbb{Z}^d} X_{k}(\omega) P_k(Y_l(\omega)\varphi_l f)
	 $,	 and similarly for $g_N^\omega$. Denote 
	 $F^\omega(t) = \pi_1 K(t)(f^\omega,g^\omega)  
	,
	 	F_N^\omega = \pi_1 K(t)(f_N^\omega,g_N^\omega).
	 $
	 , for $2<r\leq \infty$, choose $\tilde{r}, 2<\tilde{r}<r,  \frac{2}{q}+\frac{d-1}{\tilde{r}}\leq \frac{d-1}{2}, q<\infty, p\geq \max\{q,\tilde{r}\}$, for any $\varepsilon>0$, by Lemma \ref{inftytofine}, Berstein inequality, Minkowski inequality, Lemma \ref{Refinstr}, \ref{equivalentnorm},
	 \begin{align*}
	 	\|F^\omega\|_{L^1_\omega L^\infty_t L^r_x}&\leq \|F_N^\omega\|_{L^1_NL^1_\omega L^\infty_t L^r_x}\lesssim \left\|\|N^{\frac{1}{q}+d(\frac{1}{\tilde{r}}-\frac{1}{r})}F_N^\omega\|_{L_t^q L^{\tilde{r}}_x L^p_\omega}\right\|_{l^1_N}\\
	 	&\lesssim \left\|N^{\frac{1}{q}+{d(\frac{1}{\tilde{r}}-\frac{1}{r})}}\|\pi_1 K(t)(P_k(\varphi_l f), P_k(\varphi_l g))\|_{L_t^q L^r_x l^2_{k\in K_N, l}}\right\|_{l^1_N}\\
	 	&\lesssim \left\| N^{\frac{2}{q}+d(\frac{1}{\tilde{r}}-\frac{1}{r})}\|\|(P_k(\varphi_l f), P_k(\varphi_l g))\|_{\mathcal{H}^0}\|_{l^2_{k\in K_N,l}}\right\|_{l^1_N}\\
	 	&\lesssim \|\langle k \rangle^{\frac{2}{q}+d(\frac{1}{\tilde{r}}-\frac{1}{r})+\varepsilon}\|(P_k(\varphi_l f), P_k(\varphi_l)g)\|_{\mathcal{H}^0}\|_{l^2_{k,l}}\\
	 	&\lesssim\|(f,g)\|_{\mathcal{H}^{\frac{2}{q}+d(\frac{1}{\tilde{r}}-\frac{1}{r})+\varepsilon}}.
	 \end{align*}
	 For any $s>0$, we can choose $q$ large and $\tilde{r}$ closed to $r$ such that $s>\frac{2}{q}+d(\frac{1}{\tilde{r}}-\frac{1}{r})$. Then, we obtain $\|F^\omega\|_{L^\infty_t L^r_x}<\infty, ~a.e$. Choosing $r = \frac{2(d+2)}{d-2}$, $\frac{2d}{d-4}$, we have $F^\omega \in S_{loc}(\mathbb{R})\cap L^1_{loc}(\mathbb{R}, L^\frac{2d}{d-4})$. We obtain the almost sure local and global existence by Proposition \ref{localexistence}, Proposition \ref{global existence}.
   	 \end{proof}

	 \subsection{Almost sure scattering}\label{Almostsurescattering}
	 Recall the main therorm in this paper.
	 	\begin{thm}\label{finalaimrepe}
	 	For $d=4,\frac{11}{12}<s<1$ or $d = 5, \frac{15}{16}<s<1$,
	 	$(f,g)\in \mathcal{H}^s$  real.  Then, for $\omega \in \Omega, ~ a.e$, there exists a unique global solution $u$ of $(\ref{aim})$ $s.t$. 
	 	\begin{equation*}
	 		(u(t),u_t(t))- K(t)(f^\omega,g^\omega)\in C(\mathbb{R},\mathcal{H}^1), \quad u\in S(\mathbb{R}).
	 	\end{equation*}
	 	\begin{equation*}
	 		u(t) = \cos(t\langle \nabla \rangle)f^\omega+\frac{\sin t\langle \nabla \rangle}{\langle \nabla \rangle}g^\omega - \int_{0}^{t}\frac{\sin(t-s)\langle \nabla \rangle}{\langle \nabla \rangle}|u|^\frac{4}{d-2}u(s)ds
	 	\end{equation*}
	 	There exists $(u^\infty_0,u^\infty_1)\in \mathcal{H}^1$ such that
	 	\begin{equation*}
	 		\lim_{t\rightarrow \infty}\|(u(t),u_t(t))-K(t)(f^\omega,g^\omega)-K(t)(u_0^\infty,u_1^\infty)\|_{\mathcal{H}^1} = 0.
	 	\end{equation*}
	 	The similar statement also holds as $t\rightarrow -\infty$.
	 \end{thm}
 	 To prove Theorem \ref{finalaimrepe}, by Proposition \ref{uniformenergybyinduction}, we only need the following proposition.
	 \begin{pro}\label{wavepacketargumet}
	 	For $d = 4, \frac{11}{12}<s<1$ or $d = 5, \frac{15}{16}<s<1, \quad \omega$, $a.e$, there exists $N$, such that $F_{\geq N}^\omega$ satisfies the conditions of $F$ in Proposition \ref{uniformenergybyinduction}.
	 \end{pro}
 	 First, we use Proposition \ref{uniformenergybyinduction}, \ref{wavepacketargumet}  to obtain Theorem \ref{finalaimrepe}.
 	 \begin{proof}[\textbf{Proof of Theorem \ref{finalaimrepe}}]
 	 	Given $d=4, \frac{11}{12}<s<1$ or $d = 5, \frac{15}{16}<s<1$, for $\omega$, $a.e$, we choose $N$ such that $F_{\geq N}^\omega$ satisfies the condition of $F$ in Proposition \ref{uniformenergybyinduction} by Proposition \ref{wavepacketargumet}. We also have $(f_{<N}^\omega, g_{<N}^\omega)\in \mathcal{H}^1$. Thus, we obtain the solution $v$ of the equation
	 	\begin{equation*}
		   \left\{
	       \begin{array}{l}
		     v_{tt} - \Delta v + v + (v+F_{\geq N}^\omega)^\frac{d+2}{d-2} = 0, \quad (t,x)\in \mathbb{R}\times \mathbb{R}^d;\\
			 v|_{t=0} = f^\omega_{<N} , ~ v_t|_{t = 0} = g^\omega_{<N} .
	       \end{array}
		   \right.
		 \end{equation*}
	 	 By Proposition \ref{uniformenergybyinduction}, we know that $v$ is global and scatters. Since $u(t) = v(t) + F^\omega_{\geq N}(t)$ is the solution of (\ref{aim}). $\|u\|_{S(\mathbb{R})}\leq \|v\|_{S(\mathbb{R})}+\|F_{\geq N}^\omega\|_{S(\mathbb{R})}<\infty$. The scattering property of $u$ follows from the scattering property of $v$.
 	 \end{proof}
  
	 \begin{proof}[\textbf{Proof of Proposition \ref{wavepacketargumet}}]
	 We need to show that $F :=  F^\omega_{\geq N}$ satisfies (i), (ii), (iii), (iv) in Proposition \ref{uniformenergybyinduction}.  For any $s>\frac{d-2}{d+2}$,	
	 \begin{align*}
	 \|\|F^\omega_{\geq N}\|_{S(\mathbb{R})}\|_{L^1_\omega}&\lesssim \|M^\frac{d-2}{d+2}\|\|(P_k(\varphi_l f),P_k(\varphi_1 g))\|_{\mathcal{H}^0}\|_{l^2_{k\in K_M, l}}\|_{l^1_{M\geq N}}\\
	 & \sim\| M ^{-s+\frac{d-2}{d+2}}\|\|(P_k(\varphi_l f),P_k(\varphi_l g))\|_{\mathcal{H}^{s}}\|_{l^2_{K_M,l}}\|_{l^1_{M\geq N}}\\
	 &\lesssim N^{-s+\frac{d-2}{d+2}}\|(f,g)\|_{\mathcal{H}^s}.
	 \end{align*} 
 	 Then,
 	 \begin{align*}
 	 	|\{\omega:\|F_{\geq N}^\omega\|_{S(\mathbb{R})}>\eta\}|\lesssim \eta^{-1}N^{-s+\frac{d-2}{d+2}}\|(f,g)\|_{\mathcal{H}^s}.
 	 \end{align*}
	 Thus, we obtain (i) for $F= F_{\geq N}^\omega$, $ N(\omega)$ sufficient large, $\omega$ $a.e$.
	 
	 For condition (ii), similar to  the former argument, for any $s>\varepsilon>0$, choose $q<\infty, r<\frac{2d}{d-4}$ such that $\frac{2}{q}+\frac{d-1}{r}\leq \frac{d-1}{2}$, $\frac{2}{q}+d(\frac{1}{r}-\frac{d-4}{2d})<\varepsilon$, by the proof of Proposition \ref{almostsureexistence},
	 \begin{align*}
	 	\|F_{\geq N}^\omega\|_{ L^1_\omega L^\infty_t L^\frac{2d}{d-4}_x} & \lesssim \| M ^{\frac{2}{q}+d(\frac{1}{r}-\frac{d-4}{2d})}\|\|(P_k(\varphi_l f),P_k(\varphi_l g))\|_{\mathcal{H}^{0}}\|_{l^2_{K_M,l}}\|_{l^1_{M\geq N}}\\
	 	&\lesssim \|M ^{-s+\varepsilon}\|\|(P_k(\varphi_l f),P_k(\varphi_l g))\|_{\mathcal{H}^{s}}\|_{l^2_{K_M,l}}\|_{l^1_{M\geq N}}\\
	 	&\lesssim
	 	N^{-s+\varepsilon}\|(f,g)\|_{\mathcal{H}^s}
	 \end{align*}
	 Thus, for $s>0$, we have (ii) for $F= F_{\geq N}^\omega$, $N(\omega)$ sufficient large, $\omega$ $a.e$.
	   
	  For condition (iii),	for $2<r<\infty$, by Bernstein inequality, Khinchin inequality, the decay estimate \ref{refineddecayestimates}, and Lemma \ref{equivalentnorm}, we have
	 \begin{align*}
	 	&\quad \left\| \|M^\beta F^\omega_M\|_{L^1([M^{1+\theta},\infty),L^\infty_x)} \right\|_{L^1_\omega l^2_M}\\
	 	&\lesssim \left\| \|M^{\beta+\frac{d}{r}} F^\omega_M\|_{L^1([M^{1+\theta},\infty),L^r_x)} \right\|_{L^1_\omega l^2_M}\\
	 	&\lesssim  \left\|M^{\beta + \frac{d}{r}}\pi_1 K(t)(P_k(\varphi_l f),P_k(\varphi_l g))\right\|_{l^2_ML^1_t([M^{1+\theta},\infty))l^2_{k\in K_M,l}L^r_x}\\
	 	&\lesssim \left\|M^{\beta +\frac{d}{r}}\left(1+\frac{|t|}{M}\right)^{-(d-1)(\frac{1}{2}-\frac{1}{r})}P_k(\varphi_l f)\right\|_{l^2_ML^1_t([M^{1+\theta},\infty))l^2_{k\in K_M,l}L^{r'}_x}\\
	 	&\quad + \left\|M^{\beta +\frac{d}{r}}\left(1+\frac{|t|}{M}\right)^{-(d-1)(\frac{1}{2}-\frac{1}{r})}\langle\nabla\rangle^{-1}P_k(\varphi_l g)\right\|_{l^2_NL^1_t([M^{1+\theta},\infty))l^2_{k\in K_M,l}L^{r'}_x}\\
	 	&\lesssim \left\|M^{1-\frac{d-3}{2}\theta+\beta +\frac{d+d\theta-\theta}{r}}P_k(\varphi_l f)\right\|_{l^2_M l^2_{k\in K_M,l}L^{r'}}\\
	 	&\quad +\left\|M^{1-\frac{d-3}{2}\theta+\beta +\frac{d+d\theta-\theta}{r}}\langle\nabla\rangle^{-1}P_k(\varphi_l g)\right\|_{l^2_M l^2_{k\in K_M,l}L^{r'}}\\
	 	&\lesssim \|(f,g)\|_{\mathcal{H}^{1-\frac{d-3}{2}\theta+\beta + \frac{d+d\theta -\theta}{r}}}.
	 \end{align*}
	 Then for $1-\frac{d-3}{2}\theta+\beta < s$, we have for some $\varepsilon>0$,
	 \begin{equation*}
	 	\left\| \|M^{\beta+\varepsilon} F^\omega_M\|_{l^2_NL^1([N^{1+\theta},\infty),L^\infty_x)} \right\|_{L^1_\omega }\lesssim_\beta \|(f,g)\|_{\mathcal{H}^s}.
	 \end{equation*}
 	 Thus, if $\beta ,\theta ,s$ satisfy $1-\frac{d-3}{2}\theta+\beta < s$, we have that (iii) in Proposition \ref{uniformenergybyinduction} for $F^\omega_{\geq N}$, $N$ sufficient large, $\omega$ $a.e$. 
 	  
 	 For (iv), 
	 \begin{align*}
	 	F^\omega_N &= \sum_{k\in K_N}\sum_l X_{k}(\omega_1) \Bigg(\frac{1}{2}e^{it\langle \nabla \rangle}(P_k(Y_l(\omega_2)\varphi_l f)-i\langle\nabla\rangle^{-1}P_k(Y_l(\omega_2)\varphi_l g))\\
	 	&\qquad\qquad\qquad\qquad+\frac{1}{2}e^{-it\langle\nabla\rangle}(P_k(Y_l(\omega_2)\varphi_l f)+i\langle\nabla\rangle^{-1} P_k(Y_l(\omega_2)\varphi_l g))\Bigg)\\
	 	&:= F^{+,\omega}_N + F^{-,\omega}_N.
	 \end{align*}
	 Denote 
	 \begin{align}
	 	f_{k,l} = \frac{1}{2}(P_k(\varphi_lf)-i\langle\nabla\rangle^{-1}P_k(\varphi_lg)), &\quad f_{k,t_0}^{\omega_2} = \sum_l e^{it_0\langle\nabla\rangle}Y_l(\omega_2) f_{k,l},\label{definitionfkt0}\\
	 	f_k^{\omega_2} = \sum_l Y_{l}(\omega_2) f_{k,l},&\quad W_{k,\tilde{l},t_0}^{\omega_2} = \tilde{P}_k(\varphi_{\tilde{l}} {f}^{\omega_2}_{k,t_0}).\label{definiitionwavepacket}
	 \end{align}
	 Then, we have
	 \begin{align*}
	 	F^{+,\omega}_N &= \sum_{k\in K_N}X_k(\omega_1)e^{i(t-t_0)\langle \nabla\rangle} f_{k,t_0}^{\omega_2}\\	 	
	 	& = \sum_{k\in K_N} X_k(\omega_1) e^{i(t-t_0)\langle \nabla \rangle}\sum_{\tilde{l}} \tilde{P}_k(\varphi_{\tilde{l}} {f}^{\omega_2}_{k,t_0})\\
	 	&= \sum_{k\in K_N}\sum_{\tilde{l}}X_k(\omega_1)e^{i(t-t_0)\langle\nabla\rangle} W_{k,\tilde{l},t_0}^{\omega_2}.
	 \end{align*}
     To show (iv), we need the following main proposition.
 
\begin{pro}\label{difficultest}
	There exists $C(\omega,\delta)<\infty$, $a.e$, such that
	\begin{align*}
		&\quad\int_{ K_{t_0,x_0}^N} |F^{\pm,\omega}_N||u_1||u_2|^\frac{4}{d-2}\\
		&\leq C(\omega,\delta) N^{-s+\frac{d+2}{8}+2d\delta} \|u_1\|_{L^\infty_tL^2_x(K^N_{t_0,x_0})}\|u_2\|_{L^\infty_tL_x^2(K_{t_0,x_0}^N)}^\frac{d-4}{2}\\
		&\quad\quad\qquad\qquad\times\left(\|u_2\|^\frac{d(6-d)}{2(d-2)}_{L^\infty_tL^\frac{2d}{d-2}_x(K_{t_0,x_0}^N)}+(\mathcal{F}_{t_0,x_0}^N[u_2])^\frac{6-d}{4}\right).
	\end{align*}
	where
	\begin{equation*}
		{\mathcal{F}}_{t_0,x_0}^N[u] = \sup_{\substack{t':t_0\leq t'\leq t_0+N\\x':|x'-x_0|\leq 3N}}\int_{t_0}^{t_0+N}\int_{\substack{\left||x-x'|-|t-t'|\right|\leq N^{5\delta}}}|u(t,x)|^{\frac{2d}{d-2}}dxdt.
	\end{equation*}
	for $u\in L_{t,x,loc}^\frac{2d}{d-2}$.
\end{pro}
The proof of Proposition \ref{difficultest} is lengthy. We assume this proposition first, and continue to the proof of Proposition \ref{wavepacketargumet}. We still need to show (iv) in Proposition \ref{uniformenergybyinduction}.
For $M>N$, then there exists $(\tau_0,y_0)$ such that $K^M_{\tau_0,y_0}\supset K^N_{t_0,x_0}$. 

If $M\geq N^\frac{8}{d+1}$, then we have $N\leq M^\frac{d+1}{8}$. For $\varepsilon>0$, $\frac{d+2}{q}<\varepsilon$, 
\begin{align*}
	\|M^{s-\varepsilon}\|F^\omega_M\|_{L^\infty_{t,x}}\|_{L^1_\omega l^2_M}&\lesssim \|M^{s-\varepsilon+\frac{d+2}{q}}\|(P_k(\varphi_l f), P_k(\varphi_l)g)\|_{l^2_{k\in K_M,l}\mathcal{H}^0}\|_{l^2_M}\\
	&\lesssim \|(f,g)\|_{\mathcal{H}^{s-\varepsilon+\frac{d+2}{q}}}\lesssim \|(f,g)\|_{\mathcal{H}^s}.
\end{align*}
Thus, $\|F_M^\omega\|_{L^\infty_{t,x}}\lesssim C_\varepsilon(\omega)M^{\varepsilon-s}$ for some $C_{\varepsilon}(\omega)<\infty, ~a.e.$, $\omega$. Then, choose $\varepsilon = \frac{1}{8}$,
\begin{align*}
	&\quad \int_{ K_{t_0,x_0}^N}|F^\omega_M||u_1||u_2|^\frac{4}{d-2}\\&\leq N\|F^\omega_M\|_{L^\infty_{t,x}}\|u_1\|_{L^\infty_tL^2_x(K^N_{t_0,x_0})}\|u_2\|^\frac{4}{d-2}_{L^\infty_tL^\frac{8}{d-2}_x(K^N_{t_0,x_0})}\\
	&\lesssim C(\omega) M^{-s+\frac{d+2}{8}}\|u_1\|_{L^\infty_tL^2_x(K^N_{t_0,x_0})}\|u_2\|^\frac{4}{d-2}_{L^\infty_tL^\frac{8}{d-2}_x(K^N_{t_0,x_0})}.
\end{align*}

If $M< N^\frac{8}{d+1}$. Define $\tilde{u}_1:= u_1\chi_{K_{t_0,x_0}^N},~\tilde{u}_2 := u_2\chi_{K_{t_0,x_0}^N}$. Since $\frac{8}{d+1}<2$, we know that $M^{5\delta}\leq N^{10\delta}$. Then, by the definition of $\mathcal{F}_{\tau_0, y_0}^M, \tilde{u}_2$ and $\tilde{\mathcal{F}}_{t_0,x_0}^N$, we have
\begin{equation*}
	\mathcal{F}^M_{\tau_0,y_0}[\tilde{u}_2]\leq \tilde{\mathcal{F}}^N_{t_0,x_0}[u_2].
\end{equation*}
By Proposition \ref{difficultest},
\begin{align*}
	&\quad\int_{ K_{t_0,x_0}^N}|F^\omega_M||u_1||u_2|^\frac{4}{d-2}= \int_{K_{\tau_0,y_0}^M} |F^\omega_M||\tilde{u}_1||\tilde{u}_2|^\frac{4}{d-2}\\
	&\leq C(\omega,\delta)M^{-s+\frac{d+2}{8}+2d\delta} \|\tilde{u}_1\|_{L^\infty_tL^2_x(K^M_{\tau_0,y_0})}\|\tilde{u}_2\|_{L^\infty_tL_x^2(K_{\tau_0,y_0}^M)}^\frac{d-4}{2}\\
	&\quad\quad\times\left(\|\tilde{u}_2\|^\frac{d(6-d)}{2(d-2)}_{L^\infty_tL^\frac{2d}{d-2}_x(K_{\tau_0,y_0}^M)}+(\mathcal{F}_{\tau_0,y_0}^M[\tilde{u}_2])^\frac{6-d}{4}\right)\\
	&\leq C(\omega,\delta)M^{-s+\frac{d+2}{8}+C_d\delta}(\tilde{\mathcal{E}}^N_{t_0,x_0}[u])^\frac{1}{2}(\tilde{\mathcal{E}}^N_{t_0,x_0}[u])^\frac{d-4}{4}\\
	&\quad\quad\times\left((\tilde{\mathcal{E}}^N_{t_0,x_0}[u])^\frac{6-d}{4}+(\tilde{\mathcal{F}}^N_{t_0,x_0}[u])^\frac{6-d}{4}\right)\\
	&\lesssim C(\omega,\delta)M^{-s+\frac{d+2}{8}+C_d\delta}(\tilde{\mathcal{E}}^N_{t_0,x_0}[u]+\tilde{\mathcal{F}}^N_{t_0,x_0}[u])
\end{align*}
$\theta+20\delta<\alpha = s-\frac{d+2}{8}-2d\delta$, $0<\beta <s+\frac{d-3}{2}\theta-1 $, we obtain $s> \frac{d^2-d+10}{8(d-1)}$.
It means $\frac{11}{12}<s<1$, for $d = 4$; $\frac{15}{16}<s<1$, for $d = 5$.
\end{proof}

To show Proposition \ref{difficultest}, we first show the fast decay property of the linear wave in $K_{t_0,x_0}^N$ if the essential support of the initial data is away from $x_0$.
	\begin{lem}\label{outdecay}
		For $k\in K_N$, $|\tilde{l}-x_0|>CN$, $f_k = P_k f$,
		\begin{equation}\label{far}
			\left\|e^{i(t-t_0)\langle\nabla\rangle} \tilde{P}_k(\varphi_{\tilde{l}} f_k)\right\|_{L^\infty_{t,x}(K^N_{t_0,x_0})}\lesssim_L |\tilde{l} -x_0|^{-L}\|f_k\|_{L^2},\quad \forall ~L\in \mathbb{N}.
		\end{equation}
		the constant is independent to $t_0,x_0,N,k,\tilde{l}$.
	\end{lem}
	\begin{proof}[\textbf{Proof of Lemma \ref{outdecay}}]
		\begin{align*}
			&\quad e^{i(t-t_0)\langle\nabla\rangle} \tilde{P}_k(\varphi_{\tilde{l}} f_k)(x)\\
			&=\frac{1}{(2\pi)^{2d}}\int_{\mathbb{R}^d_{\eta}}\hat{f}(\eta)\varphi_k(\eta)\int_{\mathbb{R}_\xi^d}e^{ix\cdot\xi + i(t-t_0)\langle \xi \rangle} \tilde{\varphi}_k(\xi)\hat{\varphi}_{\tilde{l}}(\xi-\eta)d\xi d\eta\\
			& = \frac{e^{ix\cdot k}}{(2\pi)^{2d}}\int_{\mathbb{R}^d_{\eta}}\hat{f}(\eta)\varphi_k(\eta)e^{i\tilde{l}\cdot\eta}\int_{\mathbb{R}_\xi^d}e^{i(x-\tilde{l})\cdot\xi + i(t-t_0)\langle \xi+k \rangle} \tilde{\varphi}(\xi)\hat{\varphi}(\xi+k-\eta)d\xi d\eta.
		\end{align*}
		$|\nabla_\xi \big((x-\tilde{l})\cdot\xi + (t-t_0)\langle \xi+k \rangle\big)|\gtrsim |\tilde{l}-x_0|, \forall~(t,x)\in K_{t_0,x_0}^N$, we obtain (\ref{far}) by nonstationary phase argument.
	\end{proof}
	
	\begin{proof}[\textbf{Proof of Proposition \ref{difficultest}}] We only show the proof related to $F^{+, \omega}_N$. The argument for $F_N^{-,\omega}$ is almost the same. 
	For $C$ large enough which relies on $d$ only, 
	\begin{align*}
		&\quad\int_{ K_{t_0,x_0}^N} |F^{+,\omega}_N||u_1||u_2|^\frac{4}{d-2}\\
		&\leq \int_{ K_{t_0,x_0}^N}\left|\sum_{k\in K_N}\sum_{|\tilde{l}-x_0|>CN}X_k(\omega_1)e^{i(t-t_0)\langle\nabla\rangle} W_{k,\tilde{l},t_0}^{\omega_2}\right||u_1||u_2|^\frac{4}{d-2}\\
		&\quad +\int_{ K_{t_0,x_0}^N}\left|\sum_{k\in K_N}\sum_{|\tilde{l}-x_0|\leq CN}X_k(\omega_1)e^{i(t-t_0)\langle\nabla\rangle} W_{k,\tilde{l},t_0}^{\omega_2}\right||u_1||u_2|^\frac{4}{d-2}\\
		&:=I_1+I_2.
	\end{align*}
	For $I_1$, we use the decay estimate of $W_{k,\tilde{l}, t_0}^{\omega_2}$,
	\begin{align*}
		I_1 &\leq N \left\|\sum_{k\in K_N}\sum_{|\tilde{l}-x_0|>CN}X_k(\omega_1)e^{i(t-t_0)\langle\nabla\rangle} W_{k,\tilde{l},t_0}^{\omega_2}\right\|_{L^\infty_{t,x}(K^N_{t_0,x_0})}\\
		&\quad\quad\times\|u_1\|_{L^\infty_tL^2_x(K^N_{t_0,x_0})}\|u_2\|_{L^\infty_tL^\frac{8}{d-2}(K_{t_0,x_0}^N)}^\frac{4}{d-2}.
	\end{align*}
	By Lemma \ref{outdecay}, for any $L>\frac{3}{2}d$, recall the definitions (\ref{definitionfkt0}), (\ref{definiitionwavepacket}),
	\begin{align*}
		&\quad \left\|\sum_{N\in 2^{\mathbb{N}_0}}N^{L}\left\|\sum_{k\in K_N}\sum_{|\tilde{l}-x_0|>CN}X_k(\omega_1)e^{i(t-t_0)\langle\nabla\rangle} W_{k,\tilde{l},t_0}^{\omega_2}\right\|_{L^\infty_{t,x}(K^N_{t_0,x_0})}\right\|_{L^1_{\omega_1,\omega_2}}\\
		&\leq \sum_{N\in 2^{\mathbb{N}_0}} N^L\sum_{k\in K_N}\|X_k\|_{L^1_{\omega_1}}\sum_{|\tilde{l}-x_0|>CN}\left\|\left\|e^{i(t-t_0)\langle\nabla\rangle} W_{k,\tilde{l},t_0}^{\omega_2}\right\|_{L^\infty_{t,x}(K^N_{t_0,x_0})}\right\|_{L^1_{\omega_2}}\\
		&\lesssim_L \sum_{N\in 2^{\mathbb{N}_0}}N^{L}\sum_{k\in K_N}\sum_{|\tilde{l}-x_0|>CN}|\tilde{l}-x_0|^{-2L}\|\|{f}^{\omega_2}_{k,t_0}\|_{L^2_x}\|_{L^1_{\omega_2}}\sup_{k}\|X_k\|_{\Psi}\\
		&\lesssim_L \sum_{N\in 2^{\mathbb{N}_0}}N^{d-L} \sum_{k\in K_N} \|f_k^{\omega_2}\|_{L^2_{\omega_2}L^2}\\
		&\lesssim_L \sum_{N\in 2^{\mathbb{N}_0}}N^{\frac{3}{2}d-L}\|f_k^{\omega_2}\|_{L^2_{\omega_2}l^2_{k\in K_N}L^2}.
	\end{align*}
	Denote
	\begin{equation}\label{mindquality}
		h(\omega_2):=\|\langle k \rangle^s\|f_k^{\omega_2}\|_{ L^2_x}\|_{l^2_k}.
	\end{equation}
	Then, 
	\begin{equation}\label{basicestimateorth}
		\begin{array}{ll}
			\|h(\omega_2)\|_{L^2_\omega} &\lesssim \|\langle k \rangle^s f_{k,l}\|_{l^2_k L^2_x l^2_l}\\
			&\lesssim \|\langle k \rangle^sP_k(\varphi_l f)\|_{l^2_{k,l}L^2_x}+ \|\langle k \rangle^s \langle \nabla \rangle^{-1}P_k(\varphi_l g)\|_{l^2_{k,l}L^2_x}\\
			&\lesssim \|(f,g)\|_{\mathcal{H}^s}.
		\end{array}		
	\end{equation}
	Then, we have 
	\begin{align*}
		&\left\|\sum_{N\in 2^{\mathbb{N}_0}}N^{L}\left\|\sum_{k\in K_N}\sum_{|\tilde{l}-x_0|>CN}X_k(\omega_1)e^{i(t-t_0)\langle\nabla\rangle} W_{k,\tilde{l},t_0}^{\omega_2}\right\|_{L^\infty_{t,x}(K^N_{t_0,x_0})}\right\|_{L^1_{\omega_1,\omega_2}}\\
		&\lesssim \sum_{N\in 2^{\mathbb{N}_0}}N^{\frac{3}{2}d-L-s}\|(f,g)\|_{\mathcal{H}^s}<\infty.
	\end{align*}
	Denote
	\begin{align*}
		&\quad C(\omega,L):=
		\sum_{N\in 2^{\mathbb{N}_0}}N^{L}\left\|\sum_{k\in K_N}\sum_{|\tilde{l}-x_0|>CN}X_k(\omega_1)e^{i(t-t_0)\langle\nabla\rangle} W_{k,\tilde{l},t_0}^{\omega_2}\right\|_{L^\infty_{t,x}(K^N_{t_0,x_0})}.
	\end{align*}
	We have that $C(\omega,L)<\infty$, $a.e.$ $\omega$,  and
	\begin{align*}
		&\quad I_1\lesssim C(\omega,L) N^{1+\frac{3}{2}d-L-s}h(\omega_2) \|u_1\|_{L^\infty_tL^2_x(K^N_{t_0,x_0})}\|u_2\|_{L^\infty_tL^\frac{8}{d-2}(K_{t_0,x_0}^N)}^\frac{4}{d-2}\\
		&\lesssim N^{1+\frac{3}{2}d-L-s}C(\omega,L) \|u_1\|_{L^\infty_tL^2_x(K^N_{t_0,x_0})}\|u_2\|^\frac{d-4}{2}_{L^\infty_t L^2_x(K_{t_0,x_0}^N)}\|u_2\|^\frac{d(6-d)}{2(d-2)}_{L^\infty_tL_x^\frac{2d}{d-2}(K_{t_0,x_0}^N)}.
	\end{align*}
	It is enough to choose $L$ such that $1+\frac{3}{2}d-L<\frac{d+2}{8}$.
	
	To estimate $I_2$, we need the Bourgain's bushes argument for this problem from Bringmann \cite{Bringmann1}.
	\begin{align*}
		\mathscr{A}_{m}^{\omega_2} &=\mathscr{A}_{m,t_0,x_0,N}^{\omega_2}\\
		&: = \{(k,\tilde{l})\in \mathbb{Z}^{2d} :k\in K_N,|\tilde{l}-x_0|\leq CN, \|W_{k,\tilde{l},t_0}^{\omega_2}\|_{L^2}\in [2^m,2^{m+1}) \}.
	\end{align*}
	We have $\#\mathscr{A}_m^\omega \lesssim_d N^{2d}$. We choose $C_d>4d$.
	\begin{align*}
		I_2 &\leq \sum_{m\in \mathbb{Z}}\int_{ K_{t_0,x_0}^N}\left|\sum_{(k,l)\in \mathscr{A}^{\omega_2}_m}X_k(\omega_1)e^{i(t-t_0)\langle\nabla\rangle} W_{k,\tilde{l},t_0}^{\omega_2}\right||u_1||u_2|^\frac{4}{d-2}\\
		&\leq \sum_{2^m\leq  N^{-C_d}}\int_{ K_{t_0,x_0}^N}\left|\sum_{(k,l)\in \mathscr{A}^{\omega_2}_m}X_k(\omega)e^{i(t-t_0)\langle\nabla\rangle} W_{k,\tilde{l},t_0}^{\omega_2}\right||u_1||u_2|^\frac{4}{d-2}\\
		&\quad + \sum_{2^m>  N^{-C_d}}\int_{ K_{t_0,x_0}^N}\left|\sum_{(k,l)\in \mathscr{A}^{\omega_2}_m}X_k(\omega_1)e^{i(t-t_0)\langle\nabla\rangle} W_{k,\tilde{l},t_0}^{\omega_2}\right||u_1||u_2|^\frac{4}{d-2}\\
		&:= II_1+II_2.
	\end{align*}
	For $II_1$, we use the smallness of $W^{\omega_2}_{k,\tilde{l},t_0}$. Similar to the argument of $I_1$, we denote
	\begin{align*}
		C_1(\omega):=\sum_{N\in 2^{\mathbb{N}_0}} N^{\frac{C_d}{2}}\sum_{2^m\leq  N^{-C_d}}\left\|\sum_{(k,l)\in \mathscr{A}^{\omega_2}_{m}}X_k(\omega_1)e^{i(t-t_0)\langle\nabla\rangle} W_{k,\tilde{l},t_0}^{\omega_2}\right\|_{L^\infty_{t,x}(K^N_{t_0,x_0})}.
	\end{align*}
	Then, by Bernstein inequality, we have
	\begin{align*}
		\|C_1(\omega)\|_{L^1_{\omega_1,\omega_2}}&\lesssim \sum_{N\in 2^{\mathbb{N}_0}}N^{\frac{C_d}{2}}\sum_{2^m\leq N^{-C_d}}2^m\#\mathscr{A}_{m}^{\omega_2}\sup_{k}\|X_k\|_{\Psi}\\
		&\lesssim \sum_{N\in 2^{\mathbb{N}_0}}N^{-\frac{C_d}{2}+2d}<\infty.
	\end{align*}
	Thus,
	\begin{align*}
		II_1&\leq N \sum_{2^m\leq  N^{-C_d}}\left\|\sum_{(k,l)\in \mathscr{A}^{\omega_2}_{m}}X_k(\omega_1)e^{i(t-t_0)\langle\nabla\rangle} W_{k,\tilde{l},t_0}^{\omega_2}\right\|_{L^\infty_{t,x}(K^N_{t_0,x_0})}\\
		&\quad\quad\times\|u_1\|_{L^\infty_tL^2_x(K^N_{t_0,x_0})}\|u_2\|_{L^\infty_tL^\frac{8}{d-2}(K_{t_0,x_0}^N)}^\frac{4}{d-2}\\
		&\leq C_1(\omega)N^{1-\frac{C_d}{2}} \|u_1\|_{L^\infty_tL^2_x(K^N_{t_0,x_0})}\|u_2\|_{L^\infty_tL^\frac{8}{d-2}(K_{t_0,x_0}^N)}^\frac{4}{d-2}.
	\end{align*}
	We can enlarge $C_d$, such that $1-\frac{C_d}{2}< -s+\frac{d+2}{8}$.
	
	For $2^m > N^{-C_d}$, we need to decompose $\mathscr{A}_m^{\omega_2}$ into bushes and a nearly non-overlapping set,
	
	\begin{lem}\label{decomposebushandremainder} For  $\mu_{m}^{\omega_2} := \lceil N^\frac{d-6}{4}\#\mathscr{A}_{m}^{\omega_2}\rceil$, we have the decomposition of $\mathscr{A}_{m}^{\omega_2}$.
		\begin{align*}
			\mathscr{A}_{m}^{\omega_2} &= \bigsqcup_{j=1}^{J^{\omega_2}_{m,t_0,x_0,N}}\mathscr{B}_{j,m,t_0,x_0,N}^{\omega_2} \bigsqcup\mathscr{D}_{m,t_0,x_0,N}^{\omega_2}:= \bigsqcup_{j=1}^{J_m^{\omega_2}} \mathscr{B}_{j,m}^{\omega_2}\bigsqcup \mathscr{D}_{m}^{\omega_2}.
		\end{align*}
		$\#\mathscr{B}_{j,m}^{\omega_2}\geq \mu_{m}^{\omega_2}, ~j=1, 2, \cdots, J_{m}^{\omega_2}$. 
		\begin{itemize}
			\item For each $\mathscr{B}_{j,m}^{\omega_2}$, there exists a cube $Q$ with length $N^\delta$ such that 
			$$T_{k,\tilde{l}}\cap 2Q\neq \Phi, \quad\forall~(k,\tilde{l})\in \mathscr{B}_{j,m}^{\omega_2}.$$
			\item For each cube $Q$ with length $N^\delta$, $\#\{(k,\tilde{l})\in \mathscr{D}_{m}^{\omega_2}: T_{k,\tilde{l}}\cap 2Q\neq \Phi\}< \mu_{m}^{\omega_2}$. 
		\end{itemize}
		Denote  
		\begin{align*}
			W(\omega) &= \sup_N N^{-2d\delta} \sup_{2^m>N^{-C_d}}\sup_{\substack{t_0=0,N,\cdots,\lfloor N^\theta\rfloor N;\\x_0\in N\mathbb{Z}^d}}\sup_{j=1,2,\cdots,J^{\omega_2}_m}\\
			&\quad\quad \quad 2^{-m}(\#\mathscr{B}_{j,m}^{\omega_2})^{-\frac{1}{2}}\left\|\sum_{(k,l)\in \mathscr{B}^{\omega_2}_{j,m}}X_k(\omega_1)e^{i(t-t_0)\langle\nabla\rangle} W_{k,\tilde{l},t_0}^{\omega_2}\right\|_{L^\infty_{t,x}}\\
			&\quad + \sup_N N^{-2d\delta}\sup_{2^m>N^{-C_d}}\sup_{\substack{t_0=0,N,\cdots,\lfloor N^\theta\rfloor N;\\x_0\in N\mathbb{Z}^d}}\\
			&\quad\quad\quad 2^{-m}(\mu_m^{\omega_2})^{-\frac{1}{2}}\left\|\sum_{(k,l)\in \mathscr{D}^{\omega_2}_{m}}X_k(\omega_1)e^{i(t-t_0)\langle\nabla\rangle} W_{k,\tilde{l},t_0}^{\omega_2}\right\|_{L^\infty_{t,x}(K^N_{t_0,x_0})}
		\end{align*}
		Then, $W(\omega)<\infty$, $\omega$, $a.e$.
	\end{lem}
	We put the proof of this lemma in the Appendix, and continue to the proof of Proposition \ref{difficultest}.

	By Lemma \ref{decomposebushandremainder},
	\begin{align*}
		II_2&\leq \sum_{2^m>  N^{-C_d}} \sum_{j=1}^{J_m^{\omega_2}}\int_{ K_{t_0,x_0}^N}\left|\sum_{(k,l)\in \mathscr{B}^{\omega_2}_{j,m}}X_k(\omega_1)e^{i(t-t_0)\langle\nabla\rangle} W_{k,\tilde{l},t_0}^{\omega_2}\right||u_1||u_2|^\frac{4}{d-2}\\
		&\quad + \sum_{2^m>  N^{-C_d}} \int_{ K_{t_0,x_0}^N}\left|\sum_{(k,l)\in \mathscr{D}^{\omega_2}_{m}}X_k(\omega_1)e^{i(t-t_0)\langle\nabla\rangle} W_{k,\tilde{l},t_0}^{\omega_2}\right||u_1||u_2|^\frac{4}{d-2}\\
		&: = III_1+III_2.
	\end{align*}
	$T_{j,m,\delta}^{\omega_2} = \cup_{(k,\tilde{l})\in \mathscr{B}_{j,m}^{\omega_2}}\{(t,x): t_0\leq t\leq t_0+N, |x-(l-t\cdot \frac{k}{\langle k \rangle})|\leq N^{2\delta}\}$
	\begin{align*}
		K^N_{t_0,x_0} & = (K^N_{t_0,x_0}\cap T_{j,m,\delta}^{\omega_2}) \bigsqcup (K^N_{t_0,x_0}- T_{j,m,\delta}^{\omega_2}) := \tilde{T}_{j,m,\delta}^{\omega_2} \bigsqcup \tilde{T}_{j,m,\delta}^{\omega_2,c}.
	\end{align*}
	In $\tilde{T}_{j,m,\delta}^{\omega_2,c}$, $W_{k,\tilde{l},t_0}^{\omega_2}$ is essentially small. We can use non-stationary argument to obtain fast decay.
	\begin{lem}\label{fastdecay}For any $L>0$, there exists $W_R(\omega)<\infty$, $a.e$. such that
	\begin{align*}
		\left\|\sum_{(k,l)\in \mathscr{B}^{\omega_2}_{j,m}}X_k(\omega_1)e^{i(t-t_0)\langle\nabla\rangle} W_{k,\tilde{l},t_0}^{\omega_2}\right\|_{L^\infty_{t,x}(\tilde{T}_{j,m,\delta}^{\omega_2,c})}\leq W_R(\omega)N^{-L}2^m\#\mathscr{B}_{j,m}^{\omega_2}.
	\end{align*}
	\end{lem}
	\begin{proof}[\textbf{Proof}]
		Define
		\begin{align*}
			W_R(\omega)&:= \sup_N N^{L}\sup_{2^m>N^{-C_d}}\sup_{\substack{t_0=0,N,\cdots,\lfloor N^\theta\rfloor N;\\x_0\in N\mathbb{Z}^d}}
			\sup_{j=1,2,\cdots,J_{m}^{\omega_2}}\\
			& \qquad \qquad\frac{1}{2^m \#\mathscr{B}_{j,m}^{\omega_2}}\left\|\sum_{(k,l)\in \mathscr{B}^{\omega_2}_{j,m}}X_k(\omega_1)e^{i(t-t_0)\langle\nabla\rangle} W_{k,\tilde{l},t_0}^{\omega_2}\right\|_{L^\infty_{t,x}(\tilde{T}_{j,m,\delta}^{\omega_2,c})}.
		\end{align*}
		It is enough to show that $W_R(\omega)<\infty$, $a.e$. The argument is similar to the finiteness of $W(\omega)$. We put the proof of this in Appdendix.
	\end{proof}
	We decompose the integral in $III_1$ into two parts,
	\begin{align*}
		III_1 
		&\leq N\sum_{2^m>  N^{-C_d}}\sum_{j=1}^{J_m^{\omega_2}} \left\|\sum_{(k,l)\in \mathscr{B}^{\omega_2}_{j,m}}X_k(\omega_1)e^{i(t-t_0)\langle\nabla\rangle} W_{k,\tilde{l},t_0}^{\omega_2}\right\|_{L^\infty_{t,x}(\tilde{T}_{j,m,\delta}^{\omega_2,c})}\\
		&\quad\quad \times \|w_1\|_{L^\infty_tL^2_x(K^N_{t_0,x_0})}\|w_2\|_{L^\infty_tL_x^\frac{8}{d-2}(K_{t_0,x_0}^N)}^\frac{4}{d-2}\\
		&\quad+ N^\frac{d-2}{4}\sum_{2^m>  N^{-C_d}}\sum_{j=1}^{J_m^{\omega_2}} \left\|\sum_{(k,l)\in \mathscr{B}^{\omega_2}_{j,m}}X_k(\omega_1)e^{i(t-t_0)\langle\nabla\rangle} W_{k,\tilde{l},t_0}^{\omega_2}\right\|_{L^\infty_{t,x}(\tilde{T}_{j,m,\delta}^{\omega_2})}\\
		&\quad\quad \times \|w_1\|_{L^\infty_tL^2_x(K^N_{t_0,x_0})}\|w_2\|_{L^{\frac{16}{(d-2)(6-d)}}_tL_x^\frac{8}{d-2}(\tilde{T}_{j,m,\delta}^{\omega_2})}^\frac{4}{d-2}\\
		&:= IV_{1}+IV_{2}.
	\end{align*}
	For $IV_1$, by Lemma \ref{fastdecay}, we have
	\begin{align*}
		IV_{1}&\leq N\sum_{2^m>  N^{-C_d}} \sum_{j=1}^{J_m^{\omega_2}} C(\omega,L)N^{-L}2^m\#\mathscr{B}_{j,m}^{\omega_2}\\
		&\quad\quad\times\|u_1\|_{L^\infty_tL^2_x(K^N_{t_0,x_0})}\|u_2\|_{L^\infty_tL_x^\frac{8}{d-2}(K_{t_0,x_0}^N)}^\frac{4}{d-2}\\
		& \leq W_R(\omega)N^{1- L} \sum_{2^m>  N^{-C_d}} 2^m\#\mathscr{A}_{m}^{\omega_2}\|u_1\|_{L^\infty_tL^2_x(K^N_{t_0,x_0})}\|u_2\|_{L^\infty_tL_x^\frac{8}{d-2}(K_{t_0,x_0}^N)}^\frac{4}{d-2}.
	\end{align*}
	For $IV_2$, we use that the integral region is essentially a cone. Recall the definition of $W(\omega)$ from Lemma \ref{decomposebushandremainder}. By $J_m^{\omega_2}\leq \frac{\#\mathscr{A}_m^{\omega_2}}{\mu_m^{\omega_2}}$, $\sum_{j}\#\mathscr{B}_{j,m}^{\omega_2}\leq \#\mathscr{A}_{m}^{\omega_2}$, we have
	\begin{align*}
		IV_{2} &\leq W(\omega) N^{\frac{d-2}{4}+2d\delta}\sum_{2^m>  N^{-C_d}}\sum_{j=1}^{J_m^{\omega_2}}2^m(\#\mathscr{B}_{j,m}^{\omega_2})^\frac{1}{2}\\
		&\quad\quad\times\|u_1\|_{L^\infty_tL^2_x(K^N_{t_0,x_0})}\|u_2\|_{L^\infty_tL_x^2(K_{t_0,x_0}^N)}^\frac{d-4}{2}\|u_2\|_{L^\frac{2d}{d-2}_{t,x}(\tilde{T}^{\omega_2}_{j,m,\delta})}^\frac{d(6-d)}{2(d-2)}\\
		&\leq W(\omega) N^{\frac{d-2}{4}+2d\delta}\sum_{2^m>  N^{-C_d}}2^m\#\mathscr{A}_{m}^{\omega_2} (\mu_m^{\omega_2})^{-\frac{1}{2}}\\
		&\quad\quad\times\|u_1\|_{L^\infty_tL^2_x(K^N_{t_0,x_0})}\|u_2\|_{L^\infty_tL_x^2(K_{t_0,x_0}^N)}^\frac{d-4}{2}(\mathcal{F}_{t_0,x_0}^N[u_2])^\frac{6-d}{4}.
	\end{align*}
	For $III_2$, 
	\begin{align*}
		III_2 &\leq N\sum_{2^m>  N^{-C_d}} \left\|\sum_{(k,l)\in \mathscr{D}^{\omega_2}_{m}}X_k(\omega_1)e^{i(t-t_0)\langle\nabla\rangle} W_{k,\tilde{l},t_0}^{\omega_2}\right\|_{L^\infty_{t,x}(K_{t_0,x_0}^N)}\\
		&\quad\quad \times \|u_1\|_{L^\infty_tL^2_x(K^N_{t_0,x_0})}\|u_2\|_{L^\infty_tL_x^\frac{8}{d-2}(K_{t_0,x_0}^N)}^\frac{4}{d-2}\\
		&\leq W(\omega)N^{1+2d\delta}\sum_{2^m>  N^{-C_d}}2^m(\mu_m^{\omega_2})^{\frac{1}{2}}\\
		&\quad\quad \times \|u_1\|_{L^\infty_tL^2_x(K^N_{t_0,x_0})}\|u_2\|_{L^\infty_tL_x^2(K_{t_0,x_0}^N)}^\frac{d-4}{2}\|u_2\|^\frac{d(6-d)}{2(d-2)}_{L^\infty_tL^\frac{2d}{d-2}_x(K_{t_0,x_0}^N)}.
	\end{align*}
	By the choice of $\mu_m^{\omega_2}\lesssim N^{\frac{d-6}{4}}\#\mathscr{A}_m^{\omega_2}$ in Lemma \ref{decomposebushandremainder}, $\#\mathscr{A}_{m}^{\omega}\lesssim N^{2d}$,
	we obtain
	\begin{align*}
		II_2 &\lesssim \sum_{2^m>  N^{-C_d}}\left(W_R(\omega)N^{1-L}2^m \#\mathscr{A}_m^{\omega_2} + N^{\frac{d+2}{8}+2d\delta}W(\omega)2^m(\#\mathscr{A}_{m}^{\omega_2})^\frac{1}{2}\right)\\
		& \quad \quad\times \|u_1\|_{L^\infty_tL^2_x(K^N_{t_0,x_0})}\|w_2\|_{L^\infty_tL_x^2(K_{t_0,x_0}^N)}^\frac{d-4}{2}\\
		&\quad\quad\times\left(\|u_2\|^\frac{d(6-d)}{2(d-2)}_{L^\infty_tL^\frac{2d}{d-2}_x(K_{t_0,x_0}^N)}+(\mathcal{F}_{t_0,x_0}^N[u_2])^\frac{6-d}{4}\right)\\
		&\leq \tilde{C}(\omega,\delta)N^{\frac{d+2}{8}+2d\delta} \sum_{2^m>  N^{-C_d}}2^m(\#\mathscr{A}_m^{\omega_2})^\frac{1}{2}\\
		& \quad \quad\times \|u_1\|_{L^\infty_tL^2_x(K^N_{t_0,x_0})}\|u_2\|_{L^\infty_tL_x^2(K_{t_0,x_0}^N)}^\frac{d-4}{2}\\
		&\quad\quad\times\left(\|u_2\|^\frac{d(6-d)}{2(d-2)}_{L^\infty_tL^\frac{2d}{d-2}_x(K_{t_0,x_0}^N)}+(\mathcal{F}_{t_0,x_0}^N[w_2])^\frac{6-d}{4}\right).
	\end{align*}
	where $\tilde{C}(\omega,\delta) = C(\omega, L(\delta))+W(\omega)$ by choosing $L$ large enough (relies on $\delta$). If $\#\mathscr{A}_m^{\omega_2} \neq 0$, by (\ref{mindquality}),
	\begin{align*}
		2^m\leq \|W_{k,\tilde{l},t_0}^{\omega_2}\|_{L^2}\lesssim\|f_k^{\omega_2}\|_{L^2}\lesssim N^{-s}h(\omega_2).
	\end{align*}
	also
	\begin{align*}
	    \sum_{m\in \mathbb{Z}}2^{2m}\#\mathscr{A}_{m}^{\omega_2}&\leq \sum_{k\in K_N,\tilde{l}}\|W^{\omega_2}_{k,\tilde{l},t_0}\|_{L^2}^2
	    \lesssim \sum_{k\in K_N,\tilde{l}}\|\varphi_{\tilde{l}}{f}^{\omega_2}_{k,t_0}\|_{L^2}^2\\
	    &\lesssim \sum_{k\in K_N}\|{f}^{\omega_2}_{k,t_0}\|_{L^2}^2\lesssim\sum_{k\in K_N}\|f_k^{\omega_2}\|_{L^2}^2\lesssim N^{-2s}{h}(\omega_2)^2.
	\end{align*}
	$\#\{m: N^{-C_d}\leq 2^m\lesssim N^{-s}h(\omega_2)\}\lesssim \log N + \log \langle h(\omega_2)\rangle$.  By  Cauchy-Schwarz inequality
	\begin{align*}
		II_2 &\lesssim \tilde{C}(\omega,\delta)(\log N+\log\langle h(\omega_2)\rangle){h}(\omega_2)N^{-s+\frac{d+2}{8}+2d\delta}\\
		&\quad\quad \times \|u_1\|_{L^\infty_tL^2_x(K^N_{t_0,x_0})}\|u_2\|_{L^\infty_tL_x^2(K_{t_0,x_0}^N)}^\frac{d-4}{2}\\
		&\quad\quad\times\left(\|u_2\|^\frac{d(6-d)}{2(d-2)}_{L^\infty_tL^\frac{2d}{d-2}_x(K_{t_0,x_0}^N)}+(\mathcal{F}_{t_0,x_0}^N[u_2])^\frac{6-d}{4}\right).
	\end{align*}
	
	Combining the estimates for $I_1, II_1$, we obtain Proposition \ref{difficultest},
	with $C(\omega) \lesssim C(\omega,L)+C_1(\omega)+\tilde{C}(\omega,\delta)(\log N+\log\langle h(\omega_2)\rangle){h}(\omega_2)$. $C(\omega)<\infty$, $a.e.~\omega$.
    \end{proof}
	
	\section*{Appendix}
	\begin{proof}[\textbf{Proof of Lemma \ref{decomposebushandremainder}}]
	For the decomposition of $\mathscr{A}_{m}^{\omega_2}$, the basic idea is to extract bushes $\mathscr{B}_{j,m}^{\omega_2}$ by a greedy algorithm. We omit the argument. See Proposition 4.3 in \cite{Bringmann1} for the detailed argument. We mainly show the finiteness of $W(\omega)$.
	 
	Define
	\begin{equation*}
		B_{t_0,x_0,j,m,N}(\omega):=\frac{\chi_{j\leq J_{j,m}^{\omega_2}}}{2^{m}(\#\mathscr{B}_{j,m}^{\omega_2})^{\frac{1}{2}}}\left\|\sum_{(k,\tilde{l})\in \mathscr{B}^{\omega_2}_{j,m}}X_k(\omega_1)e^{i(t-t_0)\langle\nabla\rangle} W_{k,\tilde{l},t_0}^{\omega_2}\right\|_{L^\infty_{t,x}}.
	\end{equation*}
	\begin{equation*}
		D_{t_0,x_0,m,N}(\omega):=\frac{1}{2^{m}(\mu_m^{\omega_2})^{\frac{1}{2}}}\left\|\sum_{(k,\tilde{l})\in \mathscr{D}^{\omega_2}_{m}}X_k({\omega_1})e^{i(t-t_0)\langle\nabla\rangle} W_{k,\tilde{l},t_0}^{\omega_2}\right\|_{L^\infty_{t,x}(K^N_{t_0,x_0})}.
	\end{equation*}
	Define $\tilde{D}^K_{t_0,x_0,m,N}(\omega) = D_{t_0,x_0,m,N}(\omega)\chi_{{h}(\omega_2)\leq K}$, where $h$ is defined by (\ref{mindquality}).

	\begin{proA}\label{wavepacketfinallem}
		$B_{t_0,x_0,j,m,N}, \tilde{D}_{t_0,x_0,m,N}^K$ are sub-Gaussion random functions, and
		\begin{equation}\label{subgaussian}
			\|B_{t_0,x_0,j,m,N}\|_{\Psi}\leq CN^{d\delta},
		\end{equation}
		\begin{equation}\label{subgaussianr}
			\|\tilde{D}_{t_0,x_0,m,N}^K\|_{\Psi}\leq CKN^{d\delta},
		\end{equation}
		where $C$ is independent to $t_0,x_0,j,m,N$.
	\end{proA}
%
%
	\begin{proof}[\textbf{Proof of Proposition \ref{wavepacketfinallem}}]	
    Given $k\in K_N$, denote $\mathscr{B}_{j,m}^{\omega_2}(k):=\{ \tilde{l}:(k,\tilde{l})\in \mathscr{B}_{j,m}^{\omega_2} \}$. Since there exists a cube $Q$ with length $N^\delta$, such that $T_{k,\tilde{l}}\cap 2Q\neq \Phi, ~\forall~(k,\tilde{l})\in \mathscr{B}_{j,m}^{\omega_2}$, thus $\#\mathscr{B}_{j,m}^{\omega_2}(k)\lesssim N^{\delta d}$. Choose $q$, such that $\frac{d+2}{q}= \frac{d\delta}{2}$. For $p\geq q$, we have
	\begin{align*}
	    &\quad \|B_{t_0,x_0,j,m,N}(\omega)\|_{L^p(\Omega)}\\
	    &\lesssim N^{\frac{d+1}{q}}\left\|\frac{1}{2^{m}(\#\mathscr{B}_{j,m}^{\omega_2})^{\frac{1}{2}}}\left\|\sum_{(k,\tilde{l})\in \mathscr{B}^{\omega_2}_{j,m}}X_k(\omega_1)e^{i(t-t_0)\langle\nabla\rangle} W_{k,\tilde{l},t_0}^{\omega_2}\right\|_{L^q_{t,x}}\right\|_{L^p_{\omega_1,\omega_2}}\\
	    &\lesssim_q p^\frac{1}{2} N^{\frac{d+1}{q}}\left\|\frac{1}{2^{m}(\#\mathscr{B}_{j,m}^{\omega_2})^{\frac{1}{2}}}\left\|\sum_{\tilde{l}\in \mathscr{B}_{j,m}^{\omega_2}(k)}e^{i(t-t_0)\langle \nabla \rangle }W^{\omega_2}_{{k,\tilde{l},t_0}}\right\|_{l^2_{k\in K_N}}\right\|_{L^p_{\omega_2}L^q_{t,x}}\\
	    &\lesssim_q p^\frac{1}{2}N^\frac{d+2}{q}\left\|\frac{1}{2^{m}(\#\mathscr{B}_{j,m}^{\omega_2})^{\frac{1}{2}}}\left\|\sum_{\tilde{l}\in \mathscr{B}_{j,m}^{\omega_2}(k)}\left\|W^{\omega_2}_{{k,\tilde{l},t_0}}\right\|_{L^2}\right\|_{l^2_{k\in K_N}}\right\|_{L^p_{\omega_2}}\\
	    &\lesssim_q p^\frac{1}{2} N^\frac{d+2}{q}\left\|\frac{\|\#\mathscr{B}_{j,m}^{\omega_2}(k)\|_{l^2_{k\in K_N}}}{(\mathscr{B}_{j,m}^{\omega_2})^\frac{1}{2}}\right\|_{L^p_{\omega_2}}\\
	    &\lesssim_q p^\frac{1}{2}N^{\frac{d+2}{q}+\frac{d\delta}{2}}\lesssim p^\frac{1}{2}N^{d\delta}.
	\end{align*}
	For $p< q$, it is easy to obtain the estimate by H\"{o}lder inequality. Thus, we have $(\ref{subgaussian})$. The constant relies on $d,\delta$ only.
	
    For (\ref{subgaussianr}), similar to the former argument, denote $I_{t_0}^N: = [t_0,t_0+N]$, $\mathscr{D}_{m}^{\omega_2}(k):= \{\tilde{l}:(k,\tilde{l})\in \mathscr{D}_{m}^{\omega_2} \}$, $|\tilde{l}-x_0|\lesssim N$, thus $\#\mathscr{D}_{m}^{\omega_2}(k)\lesssim N^d$. For $2<q<\infty$,
    \begin{align*}
    	&\quad\|\tilde{D}_{t_0,x_0,j,m,N}^K\|_{L^p(\Omega)}\\
    	&\lesssim N^\frac{d+1}{q}\left\|\frac{\chi_{{h}(\omega_2)\leq K}}{2^{m}(\mu_m^{\omega_2})^{\frac{1}{2}}}\left\|\sum_{(k,\tilde{l})\in \mathscr{D}^{\omega_2}_{m}}X_k(\omega_1)e^{i(t-t_0)\langle\nabla\rangle} W_{k,\tilde{l},t_0}^{\omega_2}\right\|_{L^q_{t,x}(I_{t_0}^N\times \mathbb{R}^d)}\right\|_{L^p_{\omega_1,\omega_2}}\\
    	&\lesssim_q p^\frac{1}{2}N^{\frac{d+1}{q}}\left\|\frac{\chi_{{h}(\omega_2)\leq K}}{2^m(\mu_{m}^{\omega_2})^\frac{1}{2}}\left\|\sum_{\tilde{l}\in \mathscr{D}_m^{\omega_2}(k)}e^{i(t-t_0)\langle \nabla \rangle} W^{\omega_2}_{k,\tilde{l},t_0}\right\|_{L^q_{t,x}(I_{t_0}^N\times\mathbb{R}^d)l^2_{k\in K_N}}\right\|_{L^p_{\omega_2}}.
    \end{align*}
    To estimate the integral of $(t,x)$, we decompose 
    $$I_{t_0}^N\times \mathbb{R}^d = (I_{t_0}^N\times Q(x_0, CN))\bigsqcup (I_{t_0}^N\times Q(x_0,CN)^c),$$
    where $Q(x_0,CN)$ is the cube with center $x_0$, length $CN$. We decompose $I_{t_0}^N\times Q(x_0,CN)$ into cubes $Q$ with length $N^{\delta}$. For each cube $Q\subset I_{t_0}^N\times Q(x_0,CN)$, we decompose $\mathscr{D}_m^{\omega_2}(k) = \mathscr{S}_{1,m}^{\omega_2}(k)\cup \mathscr{S}_{2,m}^{\omega_2}(k)$, where
    \begin{equation*}
    	\mathscr{S}_{1,m}^{\omega_2}(k) = \{\tilde{l}\in \mathscr{D}_m^{\omega_2}(k): T_{k,\tilde{l}} \cap 2Q\neq \Phi \}, ~ \mathscr{S}_{2,m}^{\omega_2}(k) = \mathscr{D}_m^{\omega_2}(k)\setminus\mathscr{S}_{1,m}^{\omega_2}(k).
    \end{equation*}
    Then,
    \begin{align*}
    	&\quad\left\|\sum_{\tilde{l}\in \mathscr{S}_{1,m}^{\omega_2}(k)}e^{i(t-t_0)\langle \nabla \rangle} W^{\omega_2}_{k,\tilde{l},t_0}\right\|_{L^\infty_{t,x}(Q)l^2_{k\in K_N}}\lesssim \left\|2^m\#\mathscr{S}_{1,m}^{\omega_2}(k)\right\|_{l^2_{k\in K_N}}\\
    	&\lesssim 2^m\sup_{k}(\#\mathscr{S}_{1,m}^{\omega_2}(k))^\frac{1}{2}(\sum_{k\in K_N} \#\mathscr{S}_{1,m}^{\omega_2}(k))^\frac{1}{2}\\
    	&\lesssim 2^mN^{\frac{d\delta}{2}}(\mu_{m}^{\omega_2})^\frac{1}{2}.
    \end{align*} 
    We use the decay of $e^{i(t-t_0)\langle \nabla \rangle} W^{\omega_2}_{k,\tilde{l},t_0}$ for $\tilde{l}\in \mathscr{S}_{2,m}^{\omega_2}(k)$ in $Q$, for ${h}(\omega_2)\leq M$,
	\begin{align*}
		&\quad\left\|\sum_{\tilde{l}\in \mathscr{S}_{2,m}^{\omega_2}(k)}e^{i(t-t_0)\langle \nabla \rangle} W^{\omega_2}_{k,\tilde{l},t_0}\right\|_{L^\infty_{t,x}(Q)l^2_{k\in K_N}}\\
		&\lesssim_L2^{-m}\|N^{-\delta L}\#\mathscr{S}_{2,m}^{\omega_2}(k)\|f^{\omega_2}_{k,t_0}\|_{L^2_x}\|_{l^2_{k\in K_N}}\\
		&\lesssim_L 2^{-m}N^{-\delta 	L+d}\|f^{\omega_2}_{k}\|_{l^2_{k\in K_N}L^2_x}(\mu_m^{\omega_2})^\frac{1}{2}\\
		&\lesssim_{L}2^{-m}N^{-\delta L + d -s}M(\mu_m^{\omega_2})^\frac{1}{2}\\
		&\lesssim_{L}N^{-\delta L + d-s+C_d}M(\mu_m^{\omega_2})^\frac{1}{2}.
	\end{align*}
	Since these estimates are uniform for $Q\in I_{t_0}^N\times Q(x_0,CN)$, we obtain
	\begin{align*}
		&\quad\left\|\sum_{\tilde{l}\in \mathscr{S}_{2,m}^{\omega_2}(k)}e^{i(t-t_0)\langle \nabla \rangle} W^{\omega_2}_{k,\tilde{l},t_0}\right\|_{L^\infty_{t,x}(I_{t_0}^N\times Q(x_0,CN))l^2_{k\in K_N}}\\
		&\lesssim (2^mN^\frac{d\delta}{2}+ MN^{-L})(\mu^{\omega_2}_m)^\frac{1}{2}.
	\end{align*}
	Thus, note that $2^m>N^{-C_d}$,
    \begin{align*}
    	&\quad\left\|\frac{\chi_{{h}(\omega_2)\leq M}}{2^m(\mu_{m}^{\omega_2})^\frac{1}{2}}\left\|\sum_{\tilde{l}\in \mathscr{D}_m^{\omega_2}(k)}e^{i(t-t_0)\langle \nabla \rangle} W^{\omega_2}_{k,\tilde{l},t_0}\right\|_{L^q_{t,x}(I_{t_0}^N\times Q(x_0,CN))l^2_{k\in K_N}}\right\|_{L^p_{\omega_2}}\\
    	&\lesssim N^\frac{d+1}{q}\left\|\frac{\chi_{{h}(\omega_2)\leq M}}{2^m(\mu_{m}^{\omega_2})^\frac{1}{2}}\left\|\sum_{\tilde{l}\in \mathscr{D}_m^{\omega_2}(k)}e^{i(t-t_0)\langle \nabla \rangle} W^{\omega_2}_{k,\tilde{l},t_0}\right\|_{L^\infty_{t,x}(I_{t_0}^N\times Q(x_0,CN))l^2_{k\in K_N}}\right\|_{L^p_{\omega_2}}\\
    	&\lesssim N^\frac{d+1}{q}\left\|\frac{\chi_{{h}(\omega_2)\leq M}}{2^m(\mu_{m}^{\omega_2})^\frac{1}{2}}(2^m N^\frac{d\delta}{2}+ MN^{-L})(\mu^{\omega_2}_m)^\frac{1}{2}\right\|_{L^p_{\omega_2}}\\
    	&\lesssim N^{\frac{d+1}{q}+\frac{d\delta}{2}}+MN^{\frac{d+1}{q}-L+C_d}.
    \end{align*}
    We use the decay estimate of $e^{i(t-t_0)\langle \nabla \rangle} W^{\omega_2}_{k,\tilde{l},t_0}$  in $I_{t_0}^N\times Q(x_0,CN)^c$,
    \begin{align*}
	    &\quad\left\|\frac{\chi_{{h}(\omega_2)\leq M}}{2^m(\mu_{m}^{\omega_2})^\frac{1}{2}}\left\|\sum_{\tilde{l}\in \mathscr{D}_m^{\omega_2}(k)}e^{i(t-t_0)\langle \nabla \rangle} W^{\omega_2}_{k,\tilde{l},t_0}\right\|_{L^q_{t,x}(I_{t_0}^N\times Q(x_0,CN)^c)l^2_{k\in K_N}}\right\|_{L^p_{\omega_2}}\\
	    &\lesssim_L 2^{-m}\left\|{\chi_{{h}(\omega_2)\leq M}}\left\|\#\mathscr{D}_m^{\omega_2}(k)\|d(x,x_0)^{-L}\|_{L^q_{t,x}(V_2)}\|f_{k,t_0}^{\omega_2}\|_{L^2_x}\right\|_{l^2_{k\in K_N}}\right\|_{L^p_{\omega_2}}\\
	    &\lesssim_L 2^{-m} N^{-L+d+d+\frac{1}{q}}\left\|{\chi_{{h}(\omega_2)\leq M}}\|f^{\omega_2}_{k}\|_{l^2_{k\in K_N}L^2_x}\right\|_{L^p_{\omega_2}}\\
	    &\lesssim_{L}N^{-L+2d+\frac{1}{q}-s+C_d}M.
    \end{align*}
	Take $q$ such that $\frac{2(d+1)}{q} = \frac{d\delta}{2}$, $L = 2d+C_d$. Thus, we obtain (\ref{subgaussianr}).
	\end{proof}
	Define $B_{t_0,x_0,j,m,N}^K(\omega)=B_{t_0,x_0,j,m,N}(\omega)\chi_{{h}(\omega_2)\leq M}$
	$$W_{1,N}^K(\omega):=\sup_{2^m>N^{-C_d}}\sup_{t_0=0,N,\cdots,\lfloor N^\theta\rfloor N}\sup_{x_0\in N\mathbb{Z}^d}\sup_{j = 1, 2,\cdots,N^\frac{6-d}{4}}B_{t_0,x_0,j,m,N}^K,$$
	$$W^K_{2,N}(\omega):= \sup_{2^m>N^{-C_d}}\sup_{t_0=0,N,\cdots,\lfloor N^\theta\rfloor N}\sup_{x_0\in N\mathbb{Z}^d}D^K_{t_0,x_0,m,N}.$$
	
	$\mu_{m}^{\omega_2}\geq N^{\frac{d-6}{4}}\#\mathscr{A}_m^{\omega_2}$, thus, $\#J_{j,m}^{\omega_2}\leq \frac{\#\mathscr{A}_m^{\omega_2}}{\mu_{m}^{\omega_2}}\leq N^{\frac{6-d}{4}}$. $2^m\leq \|W_{k,\tilde{l},t_0}^{\omega_2}\|_{L^2}\lesssim \|f_k^{\omega_2}\|_{L^2}\leq N^{-s}h(\omega_2)\lesssim N^{-s}K$, thus, $\#\{m:B_{t_0,x_0,j,m,N}^K\neq 0\}\lesssim \log \frac{N^{-s}K}{N^{-C_d}}$. For fixed $N,m,t_0$, for the integral is not zero,  the set of $x_0$ is referred $X$, then
	\begin{equation*}
		\sum_{x_0\in X} \sum_{m\in \mathbb{Z}}2^{2m}\#\mathscr{A}_{m,t_0,x_0,N}^{\omega_2} \lesssim N^{-2s}{h}(\omega_2)^2\lesssim N^{-2s}K^2.
	\end{equation*} 
	Thus, $\#X\lesssim  2^{-2m} N^{-2s}K^2 \lesssim N^{2C_d-2s}K$. 
	Denote $S_{B,N}^K:= \{(m,t_0,x_0,j):B_{t_0,x_0,j,m,N}^K\neq 0\}$, $S_{D,N}^K:= \{(m,t_0,x_0): D_{t_0,x_0,m,N}^K\neq 0\}$,  then
	\begin{align*}
		\#S_{B,N}^K&\lesssim \log\frac{N^{-s}K}{N^{-C_d}}N^\theta N^{2C_d-2s}K^2N^{\frac{6-d}{4}},\\
		\#S_{D,N}^K&\lesssim \log\frac{N^{-s}K}{N^{-C_d}}N^\theta N^{2C_d-2s}K^2
	\end{align*}
	By Lemma \ref{supsubgaussianlemmma}, we obtain
	\begin{align*}
		\|W_{1,N}^K\|_{L^1}&\lesssim \log\langle\# S_{B,N}^K\rangle\sup_{(m,t_0,x_0,j)\in S_{B,N}}\|B_{t_0,x_0,j,m,N}^K\|_{\Psi}\\
		&\lesssim N^{d\delta}(\log N + \log K)\\
		\|W_{2,N}^K\|_{L^1}&\lesssim \log\langle\# S_{D,N}^K\rangle\sup_{(m,t_0,x_0)\in S_{D,N}^K}\|D_{t_0,x_0,m,N}^K\|_{\Psi}\\
		&\lesssim KN^{d\delta}(\log N + \log K).
	\end{align*}
	Thus, we obtain $W^K(\omega):=\sum_{N\in 2^{\mathbb{N}_0}}N^{-2d\delta}(W^K_{1,N}+W^K_{2,N})\in L^1(\Omega)$. Note that $W(\omega)\chi_{{h}(\omega_2)\leq K}\leq W^K(\omega)$. By ${h}(\omega_2)<\infty, W^K(\omega)<\infty, ~\forall~ K$, $a.e.$ $\omega$, we obtain $W(\omega)<\infty, ~a.e.$ $\omega$.
\end{proof}
	\begin{proof}[\textbf{Proof of Lemma \ref{fastdecay}}]
		Similar to the proof of Lemma \ref{decomposebushandremainder}, we only need to show that $W_R(\omega)\chi_{h(\omega_2)\leq K} \in 
		L^1(\Omega)$. Define
		\begin{align*}
			\tilde{B}_{t_0,x_0,j,m,N}(\omega):=\frac{\chi_{j\leq J_{j,m}^{\omega_2}}}{2^{m}\#\mathscr{B}_{j,m}^{\omega_2}}\left\|\sum_{(k,l)\in \mathscr{B}^{\omega_2}_{j,m}}X_k(\omega_1)e^{i(t-t_0)\langle\nabla\rangle} W_{k,\tilde{l},t_0}^{\omega_2}\right\|_{L^\infty_{t,x}(\tilde{T}_{j,m,\delta}^{\omega_2,c})}.
		\end{align*}
	By the argument of Lemma \ref{decomposebushandremainder}, we only need to show that for any given $L$,
	\begin{equation}\label{lll}
		\|\tilde{B}_{t_0,x_0,j,m,N}\|_{\Psi}\leq CN^{-L}.
	\end{equation}
	The constant $C$ is independent to $t_0,x_0,j,m, N$. By non-stationary argument in Lemma \ref{outdecay}, for any $L_0>0$, we have 
	$$\|e^{i(t-t_0)\langle \nabla \rangle}W_{k,\tilde{l},t_0}^{\omega}\|_{L^\infty_{t,x}(T_{j,m,\delta}^{\omega_2,c})}\lesssim N^{-\delta L_0} \|W_{k,\tilde{l},t_0}^{\omega_2}\|_{L^2}\lesssim N^{-\delta L_0}2^m.$$
	Thus,
	\begin{align*}
		\|\tilde{B}_{t_0,x_0,j,m,N}\|_{L^p}&\lesssim \left\|2^{-m}(\#\mathscr{B}_{j,m}^{\omega_2})^{-1}\sum_{(k,l)\in \mathscr{B}_{j,m}^{\omega_2}}|X_k(\omega_1)|2^mN^{-\delta L_0}\right\|_{L^p_{\omega_1,\omega_2}}\\
		&\lesssim 2^{-\delta L_0}\max_k{\|X_k\|_{L^p}}\\
		&\lesssim p^\frac{1}{2}2^{-\delta L_0}\max_k\|X_k\|_{\Psi} .
	\end{align*}
	We obtain $\|\tilde{B}_{t_0,x_0,j,m,N}\|_{\Psi}\lesssim N^{-\delta L_0}\max_{k}\|X_k\|_{\Psi}$. Choosing $L_0>\frac{L}{\delta}$, we have (\ref{lll}). Then, following the argument of almost sure finiteness of $W(\omega)$, we have $W_R(\omega)\chi_{h(\omega_2)\leq K} \in 
	L^1(\Omega)$. Thus, we obtain that $W_R(\omega)<\infty$, $a.e.$ $\omega$.
	\end{proof}
	
	\textbf{Acknowledgements:} J. Chen thanks Mingjuan Chen for her detailed  explanation of \cite{Chen_2020}, also thanks Bjoern Bringmann for detailed explanation of his randomization in \cite{Bringmann1}. Furthermore, J. Chen thanks Minjie Shan, Jia Shen and Liangchuan Wu for some helpful discussions on the wave packet decomposition, etc.
    
    \bibliographystyle{amsplain}
    \bibliography{ref}
    
    \scriptsize\textsc{Jie Chen: School of Mathematical Sciencs, Peking University, No 5. Yiheyuan Road, Beijing 100871, P.R.China}.
    
    \textit{E-mail address}: \textbf{jiechern@pku.edu.cn}
    \vspace{20pt}
    
    \scriptsize\textsc{Baoxiang Wang: School of Mathematical Sciencs, Peking University, No 5. Yiheyuan	Road, Beijing 100871, P.R.China}.
    
    \textit{E-mail address}: \textbf{wbx@pku.edu.cn}
\end{document}